\documentclass[11pt]{article}
\usepackage{hyperref}
\usepackage{amsfonts,mathrsfs,bbm,rawfonts,amsmath,amssymb,amsthm}
\usepackage{fullpage, setspace}

 \newtheorem{thm}{Theorem}
 \newtheorem{lem}[thm]{Lemma}
 \newtheorem{cor}[thm]{Corollary}
 \newtheorem{defn}[thm]{Definition}
 \newtheorem{rem}[thm]{Remark}
 \numberwithin{equation}{section}
 \numberwithin{thm}{section}

 \newcommand{\al}{\alpha}
 \newcommand{\be}{\beta}
 \newcommand{\ld}{\lambda}
 
 \newcommand{\de}{\delta}
 
 \newcommand{\ep}{\epsilon}
 \newcommand{\Si}{\Sigma}
 \newcommand{\si}{\sigma}
 \newcommand{\om}{\omega}
 \newcommand{\Om}{\Omega}
 \newcommand{\ga}{\gamma}
 \newcommand{\Ga}{\Gamma}
 \newcommand{\ka}{\kappa}
 \renewcommand{\th}{\theta}
 \newcommand{\Th}{\Theta}

 \newcommand{\E}{\mathcal{E}}
 \newcommand{\F}{\mathcal{F}}
 
 \renewcommand{\H}{\mathcal{H}}
 \newcommand{\A}{\mathscr{A}}
 \renewcommand{\S}{\mathscr{S}}
 \newcommand{\g}{\mathfrak{g}}
 \renewcommand{\k}{\mathfrak{k}}
 \newcommand{\G}{\mathscr{G}}
 \newcommand{\YMH}{\mathcal{YMH}}
 \newcommand{\YM}{\mathcal{YM}}
 \newcommand{\C}{\mathscr{C}}
 \newcommand{\D}{\mathbb{D}}
 \newcommand{\z}{\mathbf{z}}
 \newcommand{\x}{\mathbf{x}}
 \newcommand{\M}{\mathscr{M}}

 \newcommand{\Hol}{\text{Hol}}

 \newcommand{\Real}{\mathbb{R}}
 \newcommand{\Integer}{\mathbb{Z}}
 
 \newcommand{\Abs}[1]{\left\vert#1\right\vert}
 \newcommand{\abs}[1]{\vert#1\vert}
 \newcommand{\norm}[1]{\Vert#1\Vert}
 
 \def\<{\langle} \def\>{\rangle}
 \def\({\left(} \def\){\right)}
 \newcommand{\n}{\nabla}
 \newcommand{\p}{\partial}
 \renewcommand{\t}[1]{\tilde{#1}}

 \newcommand{\bd}{\bar{D}}
 \newcommand{\pa}{\p_{\th, \al}}
 \newcommand{\na}{\n_{\th, \al}}
 \newcommand{\pn}{\p_{\th, \al_n}}
 \newcommand{\X}{\mathcal{X}}

 \newcommand{\dt}[1]{\frac {d{#1}} {dt}}

\begin{document}

\title{Convergence of Yang-Mills-Higgs fields}

\author{Chong Song}

\date{\today}

\maketitle

\begin{abstract}
In this paper, we study the convergence of Yang-Mills-Higgs fields defined on fiber bundles over Riemann surfaces, where the fiber is
a compact symplectic manifold and the conformal structure of the underlying surface is allowed to vary. We show that away from the
nodes, the YMH fields converges, up to gauge, to a smooth YMH field modulo finitely many harmonic spheres, while near the nodes
where the conformal structure degenerates, the YMH fields converges to a pair consisting of a flat connection and a twisted geodesic
(with potential). In particular, we generalize the recent compactness results on both harmonic maps from surfaces and twisted holomorphic curves to general YMH fields.

\end{abstract}

\tableofcontents


\section{Introduction}

\subsection{Definition and motivation}

Suppose $\Si$ is a compact Riemann surface, $G$ is a compact connected Lie group endowed with a metric which is invariant under the adjoint action, and $\g$ is its Lie algebra . Suppose $(M, \om)$ is a compact symplectic manifold which supports a Hamiltonian action of $G$. Take an almost complex structure $J$ and a compatible metric on $M$ which is bi-invariant under the action. Let $P$ be a principal $G$-bundle on $\Si$ and $\F = P \times_G M$ be the fiber bundle associated to $P$. Then $G$ extends to an equivariant action on $\F$, which gives a moment map $\mu: \F \to P\times_{ad}\g$. Denote by $\A$ the space of smooth connections on $P$ and $\S$ the space of smooth sections of $\F$. Then for a fixed central element $c\in Z(\g)$, the \emph{Yang-Mills-Higgs(YMH) functional} for a pair (or \emph{configuration}) $(A, \phi) \in \A\times \S$ is defined by
\begin{equation*}
  \YMH(A, \phi) = \int_\Si|F_A|^2dV + \int_\Si|D_A\phi|^2dV + \int_\Si|\mu(\phi) - c|^2dV,
\end{equation*}
where $D_A$ is the covariant derivative induced by $A$ and $F_A$ is the curvature.

The critical points of the YMH functional are called \emph{Yang-Mills-Higgs fields}, which are pairs $(A,\phi)\in \A\times \S$ satisfying the following Euler-Lagrangian equation
\begin{equation}\label{e:el}
  \left\{
  \begin{aligned}
  D_A^*D_A\phi &= -\nabla\mu(\phi)(\mu(\phi)-c),\\
  D_A^* F_A &= - \phi^* D_A\phi.
  \end{aligned}
  \right.
\end{equation}

The YMH functional emerges naturally as an action functional in gauge theory, which arises in the research of electromagnetic phenomena, including the Ginzburg-Landau theory on superconductivity. The YMH fields, which stand for critical status in physics, play a central role in the theory. The YMH theory has been extensively studied in last several decades and become a cornerstone in particle physics and quantum field theories \cite{JT}. Besides its physical significance, it contains many important objects of mathematical interests as special cases. For example, the pseudo-holomorphic curves, harmonic maps and Yang-Mills fields are special YMH fields.

In this paper, we study the convergence and blow-up behavior of a sequence of YMH fields. In particular, the conformal structure of the underlying surface is allowed to vary. The motivation of our study originated from two aspects as follows.

Firstly, the compactness of YMH fields is closely related to the $\sigma$-modeland construction of symplectic invariants(cf. \cite{W}) . The Gromov-Witten(GW) invariant, which is built by the methods of holomorphic curves, has now become a basic tool in the study of global structures of symplectic manifolds. A recent development is a generalization of GW invariant, the so-called Hamiltonian GW invariants, on symplectic manifolds which admits a Hamiltonian group action\cite{CGS,M}. The construction of such an invariant is based on a proper compactification of the moduli space of \emph{twisted holomorphic maps} (also known as \emph{symplectic vortices})\cite{MT}.  More precisely, the twisted holomorphic maps is a special kind of YMH fields, namely, the minimizers of YMH functional, which are solutions to the vortex equations
\begin{equation}\label{e:vortex}
  \left\{
  \begin{aligned}
  &\bar{\partial}_A\phi = 0,\\
  &\iota_v F_A + \mu(\phi) = c.
  \end{aligned}
  \right.
\end{equation}
In \cite{MT}, Mundet i Riera and Tian give a notion of stable twisted holomorphic map and construct a compactification of the moduli space where the conformal structure is allowed to vary. This would lead to the definition of a virtual moduli cycle and a new quantum product in equivariant cohomology. For more related results in this direction, we refer to \cite{CGMS,O,X,Zi} and references therein. It is also worthy mentioning that, by the Hitchin-Kobayashi correspondence, the moduli space of vortices corresponds to the moduli space of stable holomorphic bundles \cite{B,M1,TY}. On the other hand, there do exist YMH fields which are not minimizers. See for example \cite{P1,T1,T2}. Thus a natural problem is to study the convergence of general YMH fields, which would lead to certain compactification of moduli space of general YMH fields.

Secondly, the convergence and blow-up analysis of YMH fields is of independent interest from the view of geometric analysis. In \cite{Zh}, Zhang considered YMH fields on vector bundles over compact Riemannian manifolds and showed the compactness of YMH fields up to gauge transformations. However, in our non-linear settings where the fibre is a compact manifold $M$, there are geometric obstructions for the convergence of YMH fields. This could be easily seen from the special case that, if the fiber bundle is simply $\F=\Si\times M$ and the connection $A$ is trivial, then YMH fields are just harmonic maps from surfaces $\Si$ to manifold $M$. The research of harmonic maps from surfaces has brought major impact to geometric analysis since Sacks and Uhlenbeck's pioneering work~\cite{SU}. There are numerous works related to the blow-up analysis of harmonic maps. Note that the most interesting case happens when the base manifold is a Riemann surface, as in our settings, since the energy $\norm{D_A \phi}_{L^2}$ is conformal invariant in dimension 2. A remarkable result is the bubble tree convergence of harmonic maps from surfaces shown by Parker~\cite{P}. As we will see below, the blow-up behavior of a sequence of general YMH fields brings new challenges and exhibits new phenomenons, while sharing a lot in common with harmonic maps. In fact, our result unifies many classical results covered by the general framework of YMH theory, including the compactness of holomorphic curves, twisted holomorphic curves, (gauged) harmonic maps from surfaces and Yang-Mills fields on dimension 2.

\subsection{Main Results}

Before explaining the results we achieved in this paper, let's first recall a related convergence result of YMH fields, in the special case where the conformal structure of the underlying surface is fixed. In a previous paper \cite{S}, the author used a Sacks-Uhlenbeck type $\al$-YMH functional to seek min-max YMH fields and showed the convergence of a sequence of critical points of $\al$-YMH functionals as $\al$ approaches 1. By letting $\al = 1$, we can easily obtain the convergence behavior for a sequence of YMH fields. It turns out that in this case the convergence is analogous to the bubble tree convergence of harmonic maps from surfaces (cf.\cite{P}).
\begin{thm}[\cite{S}]\label{t:fixed-metric}
Let $(\Si, h, j)$ be a fixed Riemann surface. Suppose $\{(A_n, \phi_n)\}_{n=1}^\infty\subset \A\times\S$ is a sequence of smooth YMH fields on $\Si$ with bounded YMH energy.  Then there exists a smooth YMH field $(A_\infty, \phi_\infty) \in \A\times\S$ and finitely many points $\mathbf{x} = \{x_1, x_2, \cdots, x_k\}\subset \Si$, such that $(A_n, \phi_n)$ sub-converges, up to gauge, to $(A_\infty, \phi_\infty)$ in $C^\infty_{loc}$ on $\Si\setminus\mathbf{x}$. Moreover, there exists finitely many harmonic maps $w_{ij}: S^2 \to M$ where $1\le i\le k$ and $0\le j \le l_k$, such that
\begin{equation}\label{e:energy-identity}
\lim_{n\to \infty} \YMH(A_n, \phi_n) = \YMH(A_\infty, \phi_\infty) + \sum_{i,j}\E(w_{ij}),
\end{equation}
where $\E(w_{ij}) = \norm{dw_{ij}}_{L^2}^2$ is the energy of $w_{ij}$.
\end{thm}
\begin{rem}
The last equality (\ref{e:energy-identity}) is called the energy identity which is of special interest, since it asserts that there is no energy loss during the blowing up process. Actually, a more refined analysis shows that the image of the limit section $\phi_\infty$ and the harmonic spheres $w_{ij}$ are connected, i.e. there is no neck between the bubbles. (A proof of this statement is essentially contained in the current paper as a special case.) Thus the convergence of YMH fields on a Riemann surface with fixed metric is well understood.
\end{rem}

Theorem~\ref{t:fixed-metric} would lead to a compactification of moduli space of YMH fields with fixed complex structure. However, when the complex structure $j$ of the underlying surface $\Si$ is allowed to vary, the compactification becomes more complicated and new difficulties arise.

In the special case of compactification of the moduli space of twisted holomorphic maps, two interesting phenomenons are observed in ~\cite{MT}, which do not appear in the classical Gromov compactness of holomorphic curves. First, there exist a new kind of bubbles which do not have a positive lower energy bound. Second, there might be a neck arising near the nodes, which consists of a chain of gradient flow lines of the Hamiltonian. Similar phenomenons appear in the research of harmonic maps from surfaces. During the compactification of moduli space of harmonic maps from nodal Riemann surfaces, Chen and Tian~\cite{CT} found that near the nodes, there might exist a geodesic at the neck which connects the limit map and bubbles. Zhu~\cite{Z} also considered the convergence of a sequence of harmonic maps from degenerating Riemann surfaces and showed that there is only a modified energy identity depending on the degenerating speed of conformal structure. Later, Chen, Li and Wang~\cite{CLW} gave a more refined analysis on the same problem and showed that the neck actually converges to a geodesic. Moreover, they give a formula of the length of the neck.

Now we explain the main results of this paper. Suppose $(\Si_n, h_n, j_n)$ is a sequence of Riemann surfaces which converges to a nodal surface $(\Si, h, j, \z)$ with nodes $\z$. Let $P(\Si_n)$ denote the principal $G$-bundle over $\Si_n$ and $\F(\Si_n) = P(\Si_n)\times_G M$ be the associated bundle. Let $\A(\Si_n)$ and $\S(\Si_n)$ denote the space of smooth connections on $P(\Si_n)$ and smooth sections of $\F(\Si_n)$ respectively. Suppose $\{(A_n, \phi_n)\}_{n=1}^\infty\subset \A(\Si_n)\times\S(\Si_n)$ is a sequence of YMH fields with bounded energy. Namely, the sequence satisfies the Euler-Lagrange equation~(\ref{e:el}) and there exists a constant $C$ such that $ \YMH_{h_n}(A_n, \phi_n) \le C$.

The convergence result can be divided into two parts, namely, away from the nodes and near the nodes. On every compact subset away from the nodes, since the metric converges smoothly, the convergence is very similar to the case of fixed metric guaranteed by Theorem~\ref{t:fixed-metric}, except that the singularities of the limit YMH field at the nodes cannot be removed in general.

\begin{thm}\label{t:main1}
  Suppose $(\Si_n, h_n, j_n)$ is a sequence of Riemann surfaces which converges to a nodal surface $(\Si, h, j, \z)$ with nodes $\z$. Suppose $\{(A_n, \phi_n)\}_{n=1}^\infty\subset \A(\Si_n)\times\S(\Si_n)$ is a sequence of YMH fields with bounded YMH energy. Then there exist limit fiber bundles $P(\Si\setminus \z)$ and $\F(\Si\setminus \z)$ defined over $\Si\setminus \z$, and a YMH field $(A_\infty, \phi_\infty)\in \A(\Si\setminus \z)\times \S(\Si\setminus \z)$ such that the following hold:
  \begin{enumerate}
    \item For any compact subset $\Si' \subset \Si\setminus \z$, there exist finitely many points $\mathbf{x} = \{x_1, x_2, \cdots, x_k\}\subset \Si'$ such that $(A_n, \phi_n)$ sub-converges, up to gauge, to $(A_\infty, \phi_\infty)$ in $C^\infty_{loc}$ on $\Si'\setminus \x$.
    \item There exist finitely many harmonic spheres $w_{ij}: S^2 \to M$ where $1\le i\le k$ and $0\le j \le l_k$ such that
      \begin{equation*}
          \lim_{n\to \infty} \YMH_{h_n}(A_n, \phi_n)|_{\Si'} = \YMH_h(A_\infty, \phi_\infty)|_{\Si'} + \sum_{i,j}\E(w_{ij}),
      \end{equation*}
      where $\E(w_{ij}) = \norm{dw_{ij}}_{L^2}^2$ is the energy of $w_{ij}$. Moreover, the images of the tree bubbles $w_{ij}$ and the limit section $\phi_\infty$ are connected.
    \item For each node $z\in \z$, the limit holonomy $\Hol(A_\infty, z)$ of $A_\infty$ around $z$ exists. In particular, if $\Hol(A_\infty, z) = id$, then $(A_\infty, \phi_\infty)$ can be extended over $z$ to a smooth YMH field.
  \end{enumerate}
\end{thm}
\begin{rem}
Although the singularities of the limit pair at the nodes can not be removed in general, the asymptotic behavior of $(A_\infty, \phi_\infty)$ as we approach the nodes can be traced.  See Theorem~\ref{t:limit-pair-at-node} below for more details.
\end{rem}

For the most interesting part near the nodes, it is equivalent to consider the convergence behavior of YMH fields over a sequence of cylinders whose length tends to infinity. Indeed, we will choose a family of "canonical" metrics near the nodes(see Section~\ref{s:metric}), such that the collar area near a node is isomorphic to a long cylinder $\C_n = [-T_n, T_n]\times S^1$ and the metric on $\C_n$ decays exponentially. That is, the metric has the form $g_n = \ld_n^2(dt^2+d\th^2)$ such that
\begin{equation*}
  |\ld_n(t)| \le C\de_n\exp(|t| - T_n), \quad \forall t\in [-T_n, T_n],
\end{equation*}
where $\de_n\to 0$ and $T_n= -\ln\de_n \to \infty$ as $n\to \infty$.

After restricting ourselves to the cylinders, we obtain a sequence of bundles and a sequence the YMH fields over $\C_n$ which we still denote by $\{(A_n, \phi_n)\}_{n=1}^\infty$. Since $G$ is assumed connected, the bundles are trivially and we may identify the section $\phi_n$ with a map $u_n: \C_n \to M$. By a conformal change to the standard metric $g_n\to g_0=dt^2 + d\th^2$, the problem is reduced to the convergence of $(A_n, u_n)$ over the flat cylinders $(\C_n, g_0)$ which satisfies that following conditions. First, under the standard metric $g_0$, the energy of $(A_n, u_n)$ is bounded by
\begin{equation}\label{e01}
\norm{D_{n} u_n}_{L^2(\C_n)} \le C, \ \  \norm{F_{n}}_{L^2(\C_n)} \le C\de_n,
\end{equation}
where $D_n$ denotes the derivative induced by $A_n$ and $F_n$ denotes the curvature. Second, $(A_n, u_n)$ satisfies the rescaled Euler-Lagrangian equation
\begin{equation}\label{e02}
  \left\{
  \begin{aligned}
  D^*_n D_n u_n &= -\ld_n\nabla\mu(u_n)(\mu(u_n) - c),\\
  D^*_n F_n &= -\ld_n u_n^* D_n u_n.
  \end{aligned}
  \right.
\end{equation}

To achieve desired estimates, we first choose a suitable gauge, the so-called \emph{balance temporal gauge} on the cylinder, such that the connection has the form $A_n = a_n d\th$ where $a_n \in C^\infty(\C_n, \g)$. Since dimension 2 is sub-critical for the Yang-Mills functional of the connection, the convergence of the connection $A_n$ to a limit flat connection $A_\infty$ is obvious.

On the other hand, the convergence of the map $u_n$ is more complicated due to the conformal invariance of the energy $\norm{D_nu_n}_{L^2}$ in dimension 2. Namely, there might be energy concentrations on the cylinder. There are two possibilities. The first is energy concentration near a point which give rise to a finite number of \emph{tree bubbles}, i.e. harmonic spheres. The second one is energy concentration on a "drifting" sub-cylinder(see Section~\ref{s2}). In this case, we obtain a \emph{twisted bubble}(or \emph{connecting bubble}), which is a \emph{twisted harmonic map} satisfying the equation
\begin{equation*}
D_{A_\infty}^*D_{A_\infty} v = 0.
\end{equation*}
Generally, due to the holonomy of connection $A_\infty$, the behavior of twisted bubbles is quite different from tree bubbles. However, we are still able to show that there are at most finitely many bubbles. It follows by an induction that, after finitely many steps of blowing-up's, we are left with a sequence on the cylinder where no energy concentration occurs.

After these preparations, we are now ready to state the main results on the cylinders. We denote the energy of $u_n$ on $\C_{n}$ by
$$ \E(u_n, A_n, \C_n) = \int_{\C_n}|D_n u_n|^2 dtd\th. $$
Suppose the connection has the form $A_n=a_n d\th$ in the balanced temporal gauge. Denote by $\p_{\th, a_n} := \p_\th + a_n$ the partial differential operator induced by $A_n$ along the $\th$-direction and let
\[ e_n := \int_{\{0\}\times S^1}(|\p_tu_n|^2-|\p_{\th, a_n} u_n|^2)d\th.\]
We also introduce two quantities
\[\mu:= \lim_{n\to \infty}T_n e_n, \qquad \nu:= \lim_{n\to \infty} T_n\sqrt{e_n}. \]

\begin{thm}\label{t:main3}
Suppose $\{(A_n, \phi_n)\}_{n=1}^\infty\subset \A(\C_n)\times\S(\C_n)$ is a sequence of YMH fields over flat cylinders $(\C_n, g_0)$, which satisfies~(\ref{e01}) and (\ref{e02}). Then after choosing the balanced temporal gauge and a sub-sequence, $A_n$ converges in $C^\infty_{loc}$ to a flat connection $A_\infty=\al_\infty d\th$ where $\al_\infty\in \g$ is constant. Moreover, if there is no energy concentration, then the limit of $u_n(\C_n)$ falls into the fixed point set of $\exp(2\pi \al_\infty)\in G$ and the following holds.
\begin{enumerate}
\item If $A_n$ is non-degenerating, then
 \begin{enumerate}
   \item the energy on the neck is
   \[ \lim_{n\to \infty}\E(u_n, A_n, \C_n) = 2\mu. \]
   \item the limit of $u_n(\C_n)$ is a twisted geodesic, the length of which equals to $\frac{2}{\sqrt{2\pi}}\nu$ if $\nu$ is finite; if $\nu  = +\infty$, then the neck contains an infinitely long twisted geodesic.
 \end{enumerate}
\item If $A_n$ is degenerating, then
 \begin{enumerate}
   \item the energy on the neck is
   \[ \lim_{n\to \infty}\E(u_n, A_n, \C_n) = 2\lim_{n\to \infty}\int_{\C_n}|(A_n-A_\infty)u_n|^2 + 2\mu. \]
   \item the limit of $u_n(\C_n)$ is a closed orbit of a geodesic with potential.
 \end{enumerate}
\end{enumerate}
\end{thm}

The image of a twisted geodesic is a closed orbit of a geodesic (see Definition~\ref{d:twisted-geodesic} below). Besides the energy identity, the above result gives a geometric description of the blow-up phenomenon which happens exactly at the nodes. We remark that intuitively, $e_n$ reflects the blowing-up speed of $u_n$, while $T_n$ represents the degenerating speed of node. Thus in the first case, the quantities $\mu$ and $\nu$ measures the competing of the blowing-up and degenerating processes. This is consistent with the recent result \cite{CLW} on harmonic maps. However, things become more complicated due to the interaction of the connection $A_n$. The degeneration of $A_n$ means that there is an expansion of the kernel of the operator $\p_{\th, a_n}$ in the limit(see Section\ref{s:poincare} for a detailed explanation). In this case, the degenerating speed of $A_n$ also get involved, which can be measured by the term $|A_n-A_\infty|$. It is very interesting that the neck turns out to be an orbit of a geodesic with potential, which is related to the famous C. Neumann problem(see for example\cite{R}). In fact, the geodesic with potential is a critical point of the functional
\[ E(v) = \int| d\ga|^2ds + \int (\ga, Q \ga)ds, \]
where $Q$ is a symmetric and non-negative matrix generated by the degeneration of the connection. See Section~\ref{s4} for more details.

\subsection{Perspectives}

Theorem~\ref{t:main1}  and~\ref{t:main3} provides a full picture of convergence of YMH fields over Riemann surfaces. In particular, we obtain an energy identity and a geometric description of the neck. Using these results together with the notion of stable maps due to Kontsevich, the compactification of moduli space of YMH fields with varying conformal structures follows. Although the compactification is complicated in full generality, it could be quite good in some special settings. Here we mention two important cases where our results in this paper can be applied.

\emph{Case 1: Twisted holomorphic maps.}

Recall that twisted holomorphic maps are just minimizing YMH fields which satisfy the first order equations~(\ref{e:vortex}). In particular, on the neck the equation becomes
\[ \bar{\p}_A u := \p_t u + J\p_{\th, A} u = 0. \]
In this case, it is obvious that the quantities $\mu, \nu$ defined in Theorem~\ref{t:main3} vanishes. Therefore, when the connection is non-degenerating, the energy and the length of the neck is zero. Thus the twisted geodesic reduces to a single closed orbit which is fixed by the limit holonomy of the connection. In particular, if the limit holonomy is trivial, then the orbit shrinks to a point and the neck vanishes. For example, when the Lie group is identity($G = \{id\}$) and the connection is trivial, the convergence amounts to the classical Gromov-Witten compactness of holomorphic curves. Another interesting example is when $G$ acts freely on $M$, then the holonomy has to be identity. See the recent paper \cite{V} for a similar result. The neck issue becomes more subtle when the connection is degenerating. It was shown in \cite{MT} that neck turns out to be a gradient line of the Hamiltonian and contains no energy, when the Lie group is simply $S^1$. Here we can apply Theorem~\ref{t:main3} to generalize this result to arbitrary compact connected Lie group, see Section~\ref{s:app} below.

\emph{Case 2: Fixed connection.}

If the sequence of YMH fields are chosen such that $A_n \equiv A$ for some fixed connection $A$, then there is no degeneration of the connection. Thus the non-degenerate case of Theorem~\ref{t:main3} applies and we obtain energy identity and geometric description of the neck. In particular, harmonic maps are just YMH fields with trivial connection. Thus we recover the bubble convergence of harmonic maps from (degenerating) surfaces, and provide a generalization of previously results in \cite{CT,Z,CLW}.

Finally, it is worth mentioning that, as a by-product, we obtain a Poincar\'e type inequality for flat connections on $S^1$ which generalizes the classical Poincar\'e inequality. See Lemma~\ref{l:poincare} below.

\ \\

The paper is organized as follows. Some preliminaries are given in Section~\ref{s:preliminaries}, including basic properties of YMH
fields on a fixed Riemann surface which are proved in our previous paper \cite{S}. For completeness, we also recall the construction of
canonical metrics on degenerating Riemann surfaces. Next, we consider a reduce problem on YMH fields over infinitely long cylinders.
The convergence of the connection part is proved in Section~\ref{s:connection}. The more delicate analysis and estimates for the section part are presented in
Section~\ref{s:section}. With these preparations, in Section~\ref{s:away} we show the convergence of YMH fields away from the
nodes and prove Theorem~\ref{t:main1}. Then we devote the last section to blow-up analysis of YMH
fields near the nodes and give the proof of Theorem~\ref{t:main3}. The main contribution of this paper lies in Section~\ref{s:near} where we emphasis the geometric aspects of the neck.

\section{Preliminaries}\label{s:preliminaries}

\subsection{Yang-Mills-Higgs functional}

Suppose $(\Si, j)$ is a Riemann surface and $(M, \om)$ is a symplectic manifold which supports a symplectic action of a compact connected Lie group $G$. Let $\g$ be the Lie algebra of $G$ and $\g^*$ its dual space. Suppose $P$ is a $G$-principal bundle over $\Si$ and $\F = P\times_G M$ is the associated bundle. For any element $\xi\in \g$, there is an infestimal action of $\xi$ on $M$ which generates a vector field $X_\xi \in \Ga(TM)$. Assume that the action is Hamiltonian, then there is a moment map $\mu:G \to \g^*$ such that
\[ \iota_{X_\xi}\om = d\<\mu, \xi\>. \]
If we take a bi-invariant metric on $\g$ which is invariant under the adjoint action of $G$, then $\g^*$ can be identified with $\g$ and the moment map can be regard as a map $\mu:M \to \g$. Moreover, $\mu$ is equivariant with respect to the coadjoint action on $\g^*$. That is, for any $g\in G$ and $x\in M$, we have $\mu(g\cdot x) = ad_g\mu(x)$. Thus the moment map on $M$ can be extend to $\F$ which gives a map $\mu:\F \to P\times_{ad}\g$.

Let $\A$ denote the space of smooth connections which is an affine space modeled on $\Om^1(P\times_{ad}\g)$ and $\A_{1,2}$ denote the Sobolev completion of $\A$ under Sobolev $W^{1,2}$ norm.
Let $\G := Aut(P) = P\times_{Ad}G$ be the gauge group, where $Ad$ denotes the conjugate action. A gauge transformation $s\in \G$ acts on a connection $A\in \A$ by $s(A) = s^{-1}ds + s^{-1}As$. The curvature of $A$ is a $\g$-valued two-form in $\Om^2(P\times_{ad}\g)$ defined by $F_A := dA + \frac12[A, A]$. The gauge transformation acts on $F_A$ by $s(F_A) = s^{-1}F_As$.

Let $\S = \Ga(\F)$ denote the space of smooth sections of $\F$ and $\S_{1,2}$ denote its Sobolev completion. A connection $A\in \A$ induces an covariant differential operator $D_A$ on $\S$. In fact, $A$ gives a horizontal distribution $H\subset T\F$ and induces a splitting $T\F = H\oplus T\F^v$ where $T\F^v$ is vertical, then the covariant derivative is just $D_A := \pi_A\circ d$, where $\pi_A:T\F \to T\F^v$ is the projection. In a local trivialization $\F|_U \simeq U \times M$, if we use $\{x^i\}_{i=1,2}$ to denote the local coordinates on $U\subset \Si$, then we can identify the section $\phi \in \S$ with a map $u:U \to M$ and write $A|_U = a_idx^i$ where $a_i \in C^\infty(U, \g)$, then we have
\[ (D_A \phi)|_U = du + X_{a_i}(u)dx^i. \]

By fixing a metric $h$ on $\Si$, an invariant almost complex structure $J$ on $M$ and a central element $c\in Z(\g)$, we can define the
\emph{Yang-Mills-Higgs(YMH) functional} for a pair $(A, \phi) \in A_{1,2}\times\S_{1,2}$ by
\[ \YMH_h(A, \phi) = \norm{D_A\phi}_{L^2}^2 + \norm{F_A}_{L^2}^2 + \norm{\mu(\phi) - c}_{L^2}^2. \]
The YMH functional consists of three components of independent interests, namely, the (twisted) energy functional, the Yang-Mills functional
and the Higgs potential, which we denote respectively by
$$\E(A, \phi) := \norm{D_A\phi}_{L^2}^2, \quad \YM(A) := \norm{F_A}_{L^2}^2, \quad \H(\phi) := \norm{\mu(\phi) -
c}_{L^2}^2.$$

Obviously, the YMH functional is invariant under gauge transformation, i.e. for any $s\in \G$, we have
\[ \YMH_h(s(A), s(\phi)) = \YMH_h(A, \phi). \]
Another important feature is that, for a conformal metric $\ld h, \ld>0$ on the 2 dimensional Riemann surface $\Si$, we have
\begin{equation}\label{e:conformal}
  \YMH_{\ld h}(A, \phi) = \norm{D_A \phi}_{L^2, h}^2 + \ld^{-1}\norm{F_A}_{L^2, h}^2 + \ld \norm{\mu(\phi) - c}_{L^2, h}^2.
\end{equation}

The critical points of the YMH fields is called \emph{YMH fields} which satisfy the Euler-Lagrangian equation
\begin{equation}\label{e:el1}
  \left\{
  \begin{aligned}
  D_A^*D_A\phi &= -\nabla\H(\phi),\\
  D_A^* F &= - \phi^* D_A\phi.
  \end{aligned}
  \right.
\end{equation}
Here $D_A^*$ is the dual operator of $D_A$, $\nabla\H(\phi)=(\mu(\phi)-c)\nabla \mu(\phi)$ denotes the $L^2$-gradient of $\H$ and $\phi^* D_A\phi$ acts on any $B\in \Om^1(P\times_{ad}\g)$ by
\[ (\phi^* D_A\phi, B) = (D_A\phi, B\phi). \]
In view of the conformal property (\ref{e:conformal}), after a conformal change of the metric $\ld h \to h$, the Euler-Lagrangian
equation has the form
\begin{equation*}
  \left\{
  \begin{aligned}
  D_A^*D_A\phi &= -\ld\nabla\H(\phi),\\
  D_A^* F &= - \ld\phi^* D_A\phi.
  \end{aligned}
  \right.
\end{equation*}
For a detailed deduction of the above equations, we refer to \cite{S} and \cite{L}. In the rest of this paper, we will omit the subscriptions and simply denote $D = D_A, F = F_A, \YMH = \YMH_h$ if no confusions can occur.

\subsection{epsilon regularity and removable singularity}

Let $(A, \phi)\in \A\times \S$ be a YMH field with bounded YMH energy. We first recall the Euler-Lagrangian equation for $(A,\phi)$ in local coordinates.

Taking a geodesic ball $U\subset \Si$ and a trivialization of the fiber bundles $P(U)$ and $\F(U)$, we may regard the section $\phi$ as a map $u:U \to M$ and the connection $A$ as a $\g$ valued 1-form. The exterior covariant derivative $D$ can by written as $D = d + A$ and its adjoint is simply $D^* = d^* + A^*$. Since $A$ is compatible with the metric, we have $A^*=-A$. Moreover, the curvature $F$ has the form $F = dA + \frac12[A, A]$. We embed the compact manifold $M$ into an Euclidean space $\Real^K$ and denote the second fundamental form by $\Ga$. A simple calculation (see \cite{S} for example) shows that in this trivialization, equation (\ref{e:el1}) is equivalent to
\begin{equation}\label{e:el-loc}
  \left\{
  \begin{aligned}
  &\Delta u + \Ga(u)(du, du) + d^*A\cdot u + 2A\cdot du + A^2\cdot u = \nabla\H(u_n);\\
  &d^*dA + [A, dA] + [A, [A, A]] = - u^* ( du+ Au).
  \end{aligned}
  \right.
\end{equation}

The first equation in (\ref{e:el-loc}) for the map $u$ is analogous to the equation for harmonic maps. On the other hand, the second one for the connection $A$ is not elliptic, due to the gauge invariance of the Yang-Mills functional $\YM(A)=\norm{F}_{L^2}^2$. To overcome this problem, we need to choose the \emph{Coulomb gauge}, such that $d^*A = 0$. By Uhlenbeck's theorem~\cite{Uh}, such a gauge always exists in a geodesic ball if the $L^2$-norm of the curvature is small. After fixing the Coulomb gauge, (\ref{e:el-loc}) becomes a coupled elliptic system and the techniques in the analysis for harmonic maps can be applied. Thus many results analogous to the classical results for harmonic maps essentially follows. For
example, in view of the regularity results for harmonic maps~\cite{RS,ST}, it is easy to verify that any YMH filed in $\A_{1,2}\times \S_{1,2}$ is actually smooth. Moreover, we have the following $\ep$-regularity and removable singularity theorems. For a detailed proof of the following theorems, one can refer to~\cite{S}.

\begin{lem}[$\ep$-regularity]\label{l:reg1}
Let $\D$ be the unit disk and $\D'$ be the disk with radius $\frac12$. Suppose $(A,u) \in \A(\D) \times \S(\D)$ is a smooth YMH field
with bounded YMH energy, where $\norm{F}_{L^2(\D)}$ is small and $A$ is under Coulomb gauge. \\
1)~For any $1<p<2$,
\begin{equation}\label{equ:reg0}
\norm{A}_{W^{2,p}(\D')} \le C_p(\norm{Du}_{L^2(\D)} + \norm{F}_{L^2(\D)}),
\end{equation}
where $C_p$ is a constant depending on $p$.\\
2)~There exists a constant $\epsilon_0> 0$ independent of $A$ and $u$, such that if
$$\norm{Du}_{L^2(\D)} < \epsilon_0, $$
then for any $k\ge 2$,
\begin{equation}\label{equ:reg1}
\norm{u-\bar{u}}_{W^{k,2}(\D')} + \norm{A}_{W^{k,2}(\D')} \le C_k \YMH(A,u),
\end{equation}
where $\bar{u}$ is the mean value of $u$ over $\D$ and $C_k$ is a constant depending only on $k$. In particular, we have
\begin{equation}\label{equ:reg2}
|Du(0)| + |F(0)| \le C(\norm{Du}_{L^2(\D)} + \norm{F}_{L^2(\D)}).
\end{equation}
\end{lem}

\begin{rem}
Note that the assumption of $\ep$-smallness of the energy $\norm{Du}_{L^2(\D)}$ is only needed for 2) in Lemma~\ref{l:reg1}. The $W^{2,p}$ local estimates for $A$ holds everywhere if $\norm{F}_{L^2(\D)}$ is small. In particular, in the sub-critical dimension 2, the $L^2$-norm of the curvature can always be small after a scaling of the metric.
\end{rem}

\begin{lem}[Removable singularity]\label{l:sing}
Suppose $A$ is a continuous connection on $P(\D)$ and $u\in W^{2,2}_{\text{loc}}(\mathbb{D}\setminus\{0\})$. If $(A, u)$ is a YMH field on on the punctured disk $\mathbb{D}\setminus\{0\}$ with bounded YMH energy, then $u$ can be extended to a map $\tilde{u} \in W^{2,2}(\mathbb{D})$.
\end{lem}

\subsection{Canonical metrics on degenerating surfaces}\label{s:metric}

Let $g, m$ be nonnegative integers satisfying $2g+m\ge 3$. The Deligne-Munford moduli space $\M_{g,m}$ gives a compactification
of  isomorphism classes of stable curves of genus $g$ and $m$ marked points. Given a nodal curve $(\Si, j, \z)\in \M_{g,m}$, we may
parameterize a neighborhood of $(\Si, j, \z)$ in $\M_{g,m}$ together with a family of "canonical" metrics on each element in this
neighborhood. An important feature of the metric we chose is that they are flat near the nodes. Here we follow~\cite[Section 9]{FO} to
give such a construction, which will enable us to reduce the convergence problem near the nodes to the problem on long
cylinders. For our purpose, we only illustrate such a construction locally near one node $z\in \z$.

First fix a K\"ahler metric $h$ on $(\Si, j, \z)$ such that the metric is flat in a neighborhood of $z\in \z$. Let $\pi: \tilde{\Si}\to \Si$
be
the normalization map such that $\pi(z_1) = \pi(z_2) = z$ where $z_1\in \Si_1$ and $z_2\in \Si_2$ are the preimages on the normalized
components $\Si_1, \Si_2$. There is a metric on $T_{z_1}\Si_1$ and $T_{z_2}\Si_2$ induced by the metric on $\Si_1$ and $\Si_2$,
which induces one on the tensor product $T_{z_1}\Si_1\otimes T_{z_2}\Si_2$. Thus for each nonzero element $\eta \in
T_{z_1}\Si_1\otimes T_{z_2}\Si_2$, we have a biholomorphic map $\Phi_\eta: T_{z_1}\Si_1\setminus\{0\} \to  T_{z_2}\Si_2
\setminus\{0\}$ such that
\[ w\otimes\Phi_\eta(w) = \eta, \quad \forall w\in T_{z_1}\Si_1\setminus\{0\}. \]
Now suppose $|\eta| = \de^4$ where $\de$ is a small positive number. We remove the disk $U(z_1, \de^3)$ from $\Si_1$ and $U(z_2,
\de^3)$ from $\Si_2$, where $U(x, r)$ denotes the metric ball of radius $r$ centered at $x$. Next we choose a smooth function
$\chi_\de:(0, +\infty) \to (0, +\infty)$ which is invariant by $\Phi_\eta$ and $\chi_\de(r) = 1$ if $r\ge \de^{3/2}$. Then we can glue the
two annuluses $U(z_1, \de)\setminus U(z_1, \de^3)$ and $U(z_2, \de)\setminus U(z_2, \de^3)$ together by $\Phi_\eta$ such that
$\Phi_\eta^*g_\de = g_\de$, where $g_\de = \chi_\de(r)(dr^2+r^2d\th^2)$. Thus we obtain a family of stable curves $\Si_\eta$ together
with metrics $g_\de$ decided by $\chi_\de$. We call the glued annulus in $\Si_\eta$ the \emph{collar area} and denote it by $\C_\de$.
By a transformation
$$(r, \th)\to (t,\th) := (-\log r+2\log\de, \th),$$
we find that the collar area $\C_\de$ is isomorphic to a cylinder $[-T_\de, T_\de]\times S^1$, where $T_\de=-\ln\de$ goes to infinity
as $\de\to 0$. Moreover, using the parameters $(t,\th)$, the metric $g_\de$ on the cylinder can be expressed as
\begin{equation}\label{e001}
  g_\de(t,\th) =e^{-2t}\de^4\chi_\de(e^{-t}\de^2)(dt^2+d\th^2).
\end{equation}

We fix the function $\chi_\de$ once we choose it. The above construction gives a parametrization of a neighborhood of $(\Si, j, \z)$ by
a neighborhood of the origin in $T_{z_1}\Si_1\otimes T_{z_2}\Si_2$. Thus, for any Riemann surface close to $(\Si, j, \z)$, we have a
set of collar areas together with a "canonical" metric which corresponds to the nodal set $\z$. In particular, the metric on each collar
area is almost flat.

Now suppose $(\Si_n, h_n, j_n)$ is a sequence of degenerating Riemann surfaces which converges to a nodal surface $(\Si, h, j, \z)$
where the metric $h$ is flat near the nodes and $h_n$ is chosen as above. Then the metric $h_n$ converges to $h$ on each regular
subsets while they shrink to the nodes at the collar areas. More precisely, for each node $z\in \z$ and neighborhood $U\subset \Si$,
there exists a sequence of $\de_n>0$ and collar areas $\C_n := \C_{\de_n}$, such that $(U\setminus \C_n, h_n)$ converges to
$(U\setminus z, h)$. Moreover, the collar area is isomorphic to a long cylinder $\C_n = [-T_n, T_n]\times S^1$ where $T_n\to \infty$.
We denote the restriction of $h_n$ on $\C_n$ by $g_n = \ld_n^2(dt^2+d\th^2)$. From (\ref{e001}), it is easy to see that $\ld_n$
decays exponentially when $t$ goes away from the ends to the middle of the cylinder, i.e.
\begin{equation}\label{e:exp-decay1}
  |\ld_n(t)| \le C\de_n\exp(|t| - T_n), \quad \forall t\in [-T_n, T_n],
\end{equation}
where $\de_n = e^{-T_n}\to 0$ as $n\to \infty$.

For convenience, we introduce the following definition and call $\ld_n$ is
\emph{exponentially bounded} on $\C_n$.
\begin{defn}
\begin{enumerate}
  \item A function $f(t)$ is called exponentially bounded on the interval $[-T, T]$ if there exists a constant $C>0$ such that
\[ |f(t)| \le C \exp(|t|-T), \quad \forall t\in [-T, T]. \]
  \item For any $p>1$, a map $u$ on the cylinder $\C_T := [-T, T]\times S^1$ is called $L^p$-exponentially bounded, if the function $f(t) :=
\norm{u}_{L^p(\C_t)}$ is exponentially bounded, where $\C_t = [-t, t]\times S^1 \subset \C_T$.
\end{enumerate}
\end{defn}

\section{Estimate of connection on cylinder}\label{s:connection}

\subsection{Estimate of curvature}

In this section, we derive the estimates of a connection on a cylinder. Let $G$ be a compact connected Lie group and $P$ be a principal $G$-bundle over a flat cylinder $\C_T := [-T,
T]\times S^1$. Note that $P$ can be trivialized since $G$ is connected. Suppose $A$ is a connection on $P$ which satisfies the equation
\begin{equation}\label{e:connection}
D^*F = \ld B
\end{equation}
where $B\in\Om^1(AdP)$ is $L^\infty$-bounded and $\ld(t)$ is exponentially bounded on $\C_T$ by
\begin{equation}
  |\ld(t)| \le C\de\exp(|t|-T).
\end{equation}
Suppose that the curvature $F$ satisfies, for any $t\in (0, T]$ and sub-cylinder $\C_t = [-t, t]\times S^1 \subset \C_T$,
\begin{equation}\label{e30}
\norm{F}_{L^2(\C_t)} \le C\sup_{s\in [-t, t]}\ld(s) \le C\de\exp(|t|-T).
\end{equation}
That is, $F$ is $L^2$-exponentially bounded. Moreover, we assume that the $L^2$-norm of $F$ on the cylinder is small such that Uhlenbeck's theorem~\cite{Uh} can be applied for all small balls.

\begin{lem}\label{l3}
Under the above assumptions, the curvature $F$ satisfies
\begin{equation*}
\norm{F}_{L^\infty(\C_{t-1})} \le C\de\exp(|t|-T).
\end{equation*}
\end{lem}
\begin{proof}
For any point $x \in \C_{t-1}$ and the unit disk $\D_x \subset \C_t$ centered at $x$, we may choose the Coulomb gauge on
$\D_x$ such that $D = d+A$ where $d^*A = 0$ and
\begin{equation}\label{e31}
\norm{A}_{W^{1,2}(\D_x)} \le C\norm{F}_{L^2(\D_x)}.
\end{equation}
Then $F=dA + \frac12[A, A]$ and the equation (\ref{e:connection}) becomes
\begin{equation*}
\Delta A + [A, dA] + \frac12[A, [A, A]] = \ld B.
\end{equation*}
From (\ref{e31}), it is obvious that $A$ is $L^p$-bounded for any $p<\infty$. Denote the disk centered at $x$ with radius $\frac 12$
by $\D_x'$. By a standard elliptic estimate and H\"older's inequality, we can easily deduce that for any $1<p<2$,
\[ \norm{A}_{W^{2,p}(\D_x')} \le C(\norm{A}_{W^{1,2}(\D_x)} + \norm{\ld B}_{L^p(\D_x)}). \]
It follows that $A$ belongs to $W^{1, q}$ for any $1<q<\infty$ and hence $L^\infty$ by Sobolev imbedding. Applying the elliptic
estimate again, we deduce that
\[ \norm{A}_{W^{2,q}(\D_x')} \le C(\norm{A}_{W^{1,q}(\D_x)} + \norm{\ld B}_{L^q(\D_x)}). \]
Therefore, $\nabla A$ belongs to $L^\infty$. So does $dA$ and $F$. Actually, in view of (\ref{e30}) and (\ref{e31}), we may conclude
that
\begin{equation*}
\begin{aligned}
\norm{F}_{L^\infty(\D'_x)} &\le C(\norm{A}_{W^{1,2}(\D_x)} + \norm{\ld B}_{L^\infty(\D_x)})\\
&\le C(\norm{F}_{L^2(\C_{t})} + \sup_{[-t, t]}\ld\norm{B}_{L^\infty(\D_x)})\\
&\le C\de\exp(|t|-T).
\end{aligned}
\end{equation*}
\end{proof}

\subsection{Balanced temporal gauge}

Generally, on a cylinder which is homotopic to $S^1$, there dose not exist a global Coulomb gauge even if the fiber bundle is trivial.
Thus we can not obtain global compactness of the connection from Uhlenbeck's local $L^p$-estimate. However, there do exists there is a convenient gauge on the cylinder which is referred as the \emph{balanced temporal gauge}~\cite{MT}. The existence of the balanced temporal gauge is based on the classifications of connections on bundles over $S^1$. Namely, the connections over $S^1$ are in 1-1 correspondence with the holonomy groups. Using this gauge, we can still obtain similar estimates on the whole cylinder as in Uhlenbeck's local compactness theorems, after modulo out the holonomy. To this end, we first show

\begin{lem}\label{l:gauge}
Suppose $A$ is a connection on the principle $G$-bundle $P$ over a cylinder $\C_T$, then there exists a balanced temporal gauge
such that in this gauge, we have
\[ D = d+A, \qquad A = a d\theta, \]
where $a$ is a smooth map from $\C_T$ to $\g$. Moreover, there exists a constant  $\al \in \g$ such that $$a(0, \theta) = \al, \forall \th
\in S^1.$$
\end{lem}
\begin{proof}
Assume that in a given trivialization, the connection has the form $D = d + A$ where $A = A_t dt + A_\th d\th$. We choose the desired
gauge in two steps.

First, for any fixed $\th \in S^1$, we can solve the following o.d.e. to get a family of gauge transformations $s_\th(t), \th \in S^1$ on the interval $[-T, T]$
\begin{equation*}
\left\{ \begin{aligned}
&\dt{s_\th} + A_t s_\th = 0,\\
&s_\th(0) = id.
\end{aligned} \right.
\end{equation*}
Then we define a gauge transformation $s_1(t,\th) = s_\th(t)$ on the cylinder $\C_T$ such that in this gauge where $\t{A} =
s_1^*A$. Obviously, we have
\[ \t{A}_t = g_\th^{-1}\dt{g_\th} + g_\th^{-1}A_tg_\th = 0. \]
It follows $\t{A} = \t{A}_\th d\th$ for some $\t{A}_\th \in C(\C, \g)$.

Next, the restriction $\t{A}|_{\{0\}\times S^1} = \t{A}_\th(0, \cdot)d\th$ gives a connection on the middle circle $\{0\}\times S^1$. By
the classification of connections on $S^1$, there exists a gauge transformation $s_2 \in C^1(S^1, G)$ such that $s_2^*
\t{A}|_{\{0\}\times S^1} = \al d\th$ where $\al \in \g$ is constant. Now applying $g_2$ to the whole cylinder, we have $s_2^*\t{A} = ad\th$ for some $a\in C^1(\C_T, \g)$ with desired properties.
\end{proof}

\begin{rem}\label{r1}
We may choose the balanced temporal gauge such that the constant $\al$ belongs to a compact subset of $\g$. In fact, since the Lie group $G$ is compact, there exists a compact set $\k\subset \g$ such that $\exp(2\pi \k) = G$. Due to the classification of connections on $S^1$, the connection on the middle circle $\{0\}\times S^1$ is determined by the holonomy $\Hol(A, 0)$ of $A$ around the circle. Thus we can find $\al\in \k$ such that $\exp(2\pi \al) = \Hol(A, 0)$, and hence a gauge such that $A|_{\{0\}\times S^1} = \al d\th$.
\end{rem}

\subsection{Estimate of connection}

Using the balanced temporal gauge, we give a global estimate of the connection.

\begin{lem}\label{l5}
Suppose $A$ satisfies equation~(\ref{e:connection}), $B$ belongs to $L^\infty$ and $F$ is $L^2$-exponentially bounded by (\ref{e30}). Then there exists a balanced temporal gauge and a constant $\al\in \g$, such that $A = ad\th$ and
\[ \norm{a-\al}_{W^{1,\infty}(\C_{t-1})} \le C\de\exp(|t| - T). \]
\end{lem}
\begin{proof}
By Lemma~\ref{l2}, we choose a balanced temporal gauge such that $A = a d\th$ and $a(0, \th) = \al$ is a constant. Then
we have $F = a_t dt\wedge d\th$, $D = d + ad\th$ and
$$D^* = d^* + A^* = -*d* - ad\th,$$
where $d\th$ acts by contraction. A simple calculation shows that in this gauge, equation (\ref{e:connection}) is equivalent to
\begin{equation}\label{e7-1}
  -a_{tt} d\th + (a_{t\th} + [a, a_t])dt =  \ld B.
\end{equation}
By Lemma~\ref{l3}, $F$ is $L^\infty$-exponentially bounded in $\C_{T-1}$. It follows that for all $t\in [-T+1, T-1]$, we have
\begin{equation}\label{e7-2}
 |a_t(t, \th)| \le C\sup \ld \le C\de\exp(|t| -T)
\end{equation}
and hence
\[ |a(t, \th) -\al| \le \int_0^t |a_t| dt \le C\de\exp(|t| -T). \]
By Remark~\ref{r1}, we may assume that $\al$ is bounded. So $a$ is bounded on the whole cylinder $\C_T$. On the other hand, from
equation~(\ref{e7-1}), it is obvious that
\[ |a_{t\th}| \le |[a, a_t]| + |\ld B| \le C\de\exp(|t| -T). \]
Integrating on $[0, t]$, we find that $a_\th$ is also exponentially bounded. Combining this with (\ref{e7-2}), we conclude that $a$ approaches $\al$ exponentially in $W^{1,\infty}$, as $t$ goes to $0$. This completes the proof of the lemma.
\end{proof}

\section{Estimate of section on cylinder}\label{s:section}

\subsection{Reduced equation}

Let $P$ be a trivial $G$-principle bundle over the cylinder $\C_T$ and $\F = P\times_G M$ be the associated bundle. In this section, we suppose $A$ is
a \emph{flat} connection and $u: \C_T \to M$ is a map (identified with a section $\phi\in \S(\F)$) which has finite energy
$$\E(A, u) = \norm{Du}_{L^2(\C_T)}^2 \le C $$
and satisfies the equation
\begin{equation}\label{e:section}
  D^*Du = -f
\end{equation}
where $D$ is the covariant derivative induced by $A$ and $f \in \S(u^*TM)$ is $L^\infty$-bounded.

Since the connection $A$ is flat, locally there always exists a Coulomb gauge such that $A$
vanishes. Thus we have the following $\ep$-regularity theorem, which is analogous to the one for harmonic maps.
\begin{lem}\label{l:reg2}
Suppose $A$ is a flat connection and $u\in W^{2,2}(\C_1)$ is a solution to equation~(\ref{e:section}) on the unit disk $\C_1 :=
[-1,1]\times S^1$. There exist $\ep_0>0$ such that if
\[ \norm{D u}_{L^2(\C_1)} < \ep_0, \]
then
\begin{equation}\label{e:reg}
\norm{D u}_{C^0(\C_{\frac12})} \le C(\norm{D u}_{L^2(\C_1)} + \norm{f}_{L^p(\C_1)}),
\end{equation}
where $\C_{\frac12} := [-\frac12, \frac12]\times S^1$ is a sub-cylinder and $p>2$.
\end{lem}
\begin{proof}
For each point $x\in \C_{\frac12}$, we may choose a disk $\D(x,\frac12)$ centered at $x$ with radius $\frac12$ and a trivialization
which puts the connection $A$ in Coulomb gauge. Since $A$ is flat, it vanishes in this gauge and the derivative is simply $D=\n$, which
is the Levi-Civita connection induced by the metric. Thus if we embed $M$ into an Euclidean space $\Real^K$ and denote the second
fundamental form by $\Ga$, equation (\ref{e:section}) becomes
\[ \tau(u) = \Delta u + \Ga(u)(\n u, \n u) = f, \]
where $\tau(u)$ is the so-called tension field. Then by $\ep$-regularity for harmonic maps (see~\cite{DT} for example), we can show that
\begin{equation*}
\norm{u - \bar u}_{W^{2,2}(\D')} \le C(\norm{\n u}_{L^2(\D)} + \norm{f}_{L^2(\D)}).
\end{equation*}
It follows by Sobolev embedding that $\n u\in L^{2p}$ for any $p>2$. This in turn implies $\Delta u \in L^p$. Thus by the standard
$L^p$ estimate, we have
\begin{equation*}
\norm{u - \bar u}_{W^{2,p}(\D')} \le C(\norm{\n u}_{L^2(\D)} + \norm{f}_{L^p(\D)}).
\end{equation*}
Then the embedding $W^{2,p}\hookrightarrow C^1$ gives
\begin{equation*}
|\nabla u(x)| \le C(\norm{\n u}_{L^2(\C_1)} + \norm{f}_{L^p(\C_1)}),
\end{equation*}
or equivalently
\begin{equation}\label{e11}
|D u(x)| \le C(\norm{D u}_{L^2(\C_1)} + \norm{f}_{L^p(\C_1)}).
\end{equation}
Note that the above estimate is gauge equivalent and hence dose not depend on the choice of Coulomb gauge. Therefore, (\ref{e11})
holds for any $x\in \C_{\frac12}$ and any choice of gauge.
\end{proof}

The Coulomb gauge only exists locally. To establish global estimates, we will always choose the temporal gauge in the following context. Since the
connection is flat, by Lemma~\ref{l:gauge} we may choose a balanced temporal gauge such that $A =\al d\th$ where $\al\in \g$ is a
constant. It follows that in this gauge $d^*A = 0$. A simple calculation yields
\[ D^*Du = (d^*+A^*)(d+A)u = d^*du -2Adu - A^2u. \]
Then equation (\ref{e:section}) is equivalent to
\begin{equation}\label{e1}
\tau(u) + 2\al\cdot u_\th + \al^2\cdot u = f.
\end{equation}

To understand the above equation better, we first give the explicit expression of the action of $\al\in \g$ on $M$. Denote the 1-parameter
group of isomorphisms generated by $\al\in \g$ by $\Phi_s(y) = \exp_y(s\al):M\to M$ for $s\in \Real$, then the infinitesimal action
of $\al$ on $M$ corresponds to a vector field $X \in \Ga(TM)$ given by
\[ \al\cdot y = \frac{d}{ds}\Big|_{s=0}\Phi_s(y) = X(y), ~\forall y\in M. \]
Similarly, $\al$ acts on a vector field $V\in \Ga(TM)$ by
\[ \al\cdot V = \frac{\n}{ds}\Big|_{s=0}d\Phi_s(V) = \nabla X \cdot V,  \]
where $\n$ denotes the Levi-Civita connection induced by the metric on $M$. Note that since $A$ is compatible with the metric, $\n X$ is skew-symmetric, i.e.
\[ \n X(V, W) = -\n X(W, V), ~V, W\in \Ga(TM). \]
Using the vector field $X$, we can write equation (\ref{e1}) as
\begin{equation}\label{e2}
  \tau(u) + 2 \nabla X(u)\cdot u_\theta + \nabla X(u)\cdot X(u) = f.
\end{equation}
One can write equation~(\ref{e2}) in a more compact form. Actually, we can define the partial covariant differential operator in the
$\th$-direction induced by $\al$ by
\[\na := \n_\th + \al, \]
then the connection $D$ splits into
\[ D = \n + \al d\th = \n_t dt + \na d\th. \]
In particular, for a map $u\in C^\infty(\C_T, M)$ we have
\[ \na u =  \p_\th u + X(u), \]
and for a vector field $V \in\Ga(u^*TM)$
\[\na V = \n_\th V + \n X\cdot V. \]
With these notations, we can rewrite equation (\ref{e1}) as
\begin{equation}\label{e7-3}
\n_t^2 u + \na^2 u = f.
\end{equation}

Next we embed the compact manifold $M$ into an Euclidean space. By an equivariant version of Nash's embedding theorem proved by
Moore and Schlafly~\cite{MS}, we can take the embedding to be equivariant under the group action.

\begin{thm}\label{t:MS}\cite{MS}
Suppose the compact Lie group $G$ acts on the compact symplectic manifold $M$ (which is equipped with an equivariant metric) by isometries,
then there exist an
orthogonal representation $\iota:G\to O(K)$ and an embedding from $M$ to $\Real^K$ which is equivariant with respect to $\iota$.
\end{thm}

Therefore, using this representation, the Lie algebra $\g$ corresponds to skew-symmetric $K\times K$ matrices. If we denote the
skew-symmetric matrix corresponding to $\al \in \g$ by $\X := \iota(\al)$, then the infinitesimal action of $\al$ on a point $y\in M
\hookrightarrow\Real^K$ is simply
\[ \al\cdot y = X(y) = \X\cdot y \]
where the $\cdot$ on the right denotes the multiplication by the matric $\X$ on the vector $y\in \Real^K$. On the other hand, since $dX
= \X$, it follows that the action of $\al$ on a vector field $V\in \Ga(TM)$ is given by
\[ \al\cdot V = \n X\cdot V = (\X\cdot V)^{\top}, \]
where $\top$ denotes the projection from $\Real^K$ to the tangent space of $M$. More precisely, if we denote the second fundamental
form of $M$ by $\Ga$, then
\[ \nabla X = d X + \Ga\cdot X = \X + \Ga\cdot X. \]
Consequently, we can write equation (\ref{e1}) as
\begin{equation*}
  \tau(u) + 2(\X\cdot u_\theta)^{\top} + (\X\cdot X(u))^{\top} = f,
\end{equation*}
or
\begin{equation}\label{e002}
  \tau(u) + 2\X\cdot u_\theta + \X\cdot X(u) + 2\Ga(u)(u_\th, X(u)) + \Ga(u)(X(u), X(u)) = f.
\end{equation}
We also have a compact form of equation~(\ref{e002}). In fact, we may define a partial differential operator
\begin{equation}\label{e:partial-derivative}
  \pa := \p_\th + \X.
\end{equation}
Then for a map $u$, $\pa  u = \na u$ is identical while for a vector field $V$, we have
$$\na V = (\pa V)^\top = \pa V + \Ga(X, V).$$
Thus we can rewrite equation (\ref{e:section}), or equivalently (\ref{e7-3}) as
\begin{equation}\label{e3}
  \p_t^2u + \pa ^2u + \Ga(u)(Du, Du) = f.
\end{equation}

For convenience, we will identify $\al$ with $\X$ and always think of $\al\in \g$ as a skew-symmetric matrix in the sequel.

\subsection{Poincar\'e inequality for flat connection}\label{s:poincare}

To obtain desired energy estimates for harmonic maps, a key tool is the Poincar\'e inequality. Recall that for any map $u\in
W^{2,2}(S^1, \Real^K)$, we have the Poincar\'e inequality
\[ \int_{S^1}|u_\th|^2d\th \le C \int_{S^1}|u_{\th\th}|^2d\th. \]
Here we show that the Poincar\'e inequality also holds for a flat connection $A$.

First recall that through parallel transportation, a flat connection $A$ over the cylinder can be identified with a connection over $S^1$,
which we still denote by $A$. Suppose the holonomy of $A$ around $S^1$ is $\Hol(A)$, then there exists a trivialization such that $A = \al
d\th$ and $\Hol(A) = \exp(2\pi \al)$, where $\al \in \g$(which lies in a compact subset $\k$ by Remark~\ref{r1}). Here, by using the
equivariant embedding and representation $\iota$ given by Theorem~\ref{t:MS}, we think of $\al$ as a skew-symmetric $K\times K$ matrix. We use $\pa =\p_\th +\al$ to denote the derivative on $S^1$ induced by $A$.

\begin{lem}\label{l:poincare}
Suppose $A=\al d\th$ is a flat connection on the cylinder. Then there exist a constant $C_A$ only depending on
$A$, such that the Poincar\'e type inequality
\begin{equation}\label{e:poincare}
  \int_{S^1}|\pa u|^2d\th \le C_A \int_{S^1}|\pa ^2 u|^2d\th
\end{equation}
holds for all maps $u\in W^{2,2}(S^1, \Real^K)$.
\end{lem}
\begin{proof}
We prove by contradiction. Suppose the inequality (\ref{e:poincare}) is not true, then there exists $u_n\in W^{2,2}(S^1)$ such
that
\[\int_{S^1}|\pa u_n|^2d\th \ge n\int_{S^1}|\pa ^2 u_n|^2d\th. \]
By a rescaling, we may normalize $\int_{S^1}|\pa u_n|^2d\th = 1$. Then $\int_{S^1}|\pa ^2 u_n|^2d\th\le 1/n$. Then one easily checks
that $\norm{\pa u_n}_{W^{1,2}}\le C$. Therefore, $v_n:=\pa u_n$ sub-converges to some $v$ weakly in $W^{1,2}$ which satisfies
\[ \int_{S^1}|\pa v|^2d\th\le \liminf_{n\to \infty}\int_{S^1}|\pa ^2 u_n|^2d\th = 0. \]
Hence $\pa v=0$. On the other hand, by Sobolev embedding, $v_n$ converges strongly to the same map $v$ strongly in $L^{2}$. Thus,
\[ \int_{S^1}|v|^2d\th = \lim_{n\to \infty}\int_{S^1}|v_n|^2d\th = \lim_{n\to \infty}\int_{S^1}|\pa ^2 u_n|^2d\th = 1. \]
However, we also have
\[ \int_{S^1}(v_n, v)d\th = \int_{S^1}(\pa u_n, v)d\th = -\int_{S^1}(u_n, \pa v)d\th = 0. \]
Taking limit, we find that $\int_{S^1}|v|^2d\th=0$. A contradiction.
\end{proof}

A key difference of the above Poincar\'e inequality (\ref{e:poincare}) from the classical one is that the Poincar\'e constant $C_A$ depends on the flat connection $A$. Generally, the Poincar\'e constant can \emph{not} be chosen uniformly. This is one of the difficulties in proving the energy identity in Section~\ref{s:near} below.

To trace the dependence of $C_A$ on $A$, we study the elliptic operator
\[ L_A = -\pa^2: H^2(S^1)\to L^2(S^1). \]
Since $A$ is compatible with the metric, $L_A$ is self-adjoint and non-negative. Lemma~\ref{l:poincare} actually shows that the spectra of $L_A$ can not accumulate at zero for a fixed connection $A$. Thus we may let $\si_A^2>0$ be the first positive eigenvalue of $L_A$ and $v \in H^2(S^1)$ be an
eigenvector. If there exists $u$ such that $\pa u = v$, then $-\pa^3 u = \si_A^2 v$. It follows
\[  \int_{S^1}|\pa^2 u|^2d\th = -\int_{S^1}(\pa^3u, \pa u)d\th = \si_A^2 \int_{S^1}|\pa u|^2d\th. \]
Consequently, the Poincar\'e constant in~(\ref{e:poincare}) is just $C_A=1/ \si_A^2$ and can be arbitrarily large if $\si_A$ is close to 0. In fact, if $\si_A$ goes to zero as $A$ varies, there might be a jump of the fist positive eigenvalue of $L_A$. Or equivalently, the kernel of $L_A$ expands in this case. This explains why and when the constant $C_A$ may go to infinity. Also note that $\pa ^2 u = 0$ implies $\pa  u=0$, since we have
\begin{equation}\label{e:10}
 \int_{S^1}|\pa u|^2 = - \int_{S^1}(\pa^2 u, u)d\th.
\end{equation}
Thus we are led to the following definition.

\begin{defn}\label{d:degenerate}
Suppose $A_n=\al_n d\th \in \A(S^1), n\in \Integer$ is a sequence of flat connections which converges to $A=\al d\th \in \A(S^1)$ in $W^{1,2}$. The convergence of $A_n$ is called non-degenerating, if there exists $N\in \Integer$, such that for all $u\in \ker(\p_{\th, \al})$, we have $u\in \ker(\p_{\th, \al_n})$, $\forall n\ge N$. Otherwise, we say that the sequence is degenerating.
\end{defn}

An easy consequence is that the Poincar\'e constant is uniformly bounded if the convergence of $A_n$ is non-degenerating.

\begin{lem}\label{l:uniform}
Suppose $\{A_n\}_{n=1}^{\infty}\subset \A(S^1)$ is a non-degenerating sequence of connections, then the corresponding Poincar\'e constant $C_{A_n}$ given by Lemma~\ref{l:poincare} is uniformly bounded for sufficiently large $n$.
\end{lem}
\begin{proof}
We follow a similar argument as in the proof of Lemma~\ref{l:poincare}. Assume that there exists a sequence of maps $u_n \in W^{2,2}(S^1)$, such that $\lim_{n\to \infty}A_n =A$ and
\[\int_{S^1}|\p_{\th, \al_n} u_n|^2d\th \ge n\int_{S^1}|\p_{\th, \al_n}^2 u_n|^2d\th. \]
After rescaling, we may assume that for $v_n:= \p_{\th, \al_n} u_n$ we have $\int_{S^1}|v_n|^2d\th=1$ and $\int_{S^1}|\p_{\th, \al_n}v_n|^2d\th\le 1/n$. Then one easily verifies that $v_n$ sub-converges to some map $v$ weakly in $W^{1,2}$ and strongly in $L^2$. Therefore, one can deduce that $\pa v=0$.

Since $A_n$ is non-degenerate, we have $\p_{\th, \al_n}v=0$ for sufficiently large $n$. It follows
$$\int_{S^1}|v|^2d\th = \lim_{n\to \infty}\int_{S^1}(v_n, v)d\th = \lim_{n\to \infty}\int_{S^1}|u_n, \p_{\th, \al_n}v|^2d\th = 0.$$
However, this contradicts to the fact that
$$\int_{S^1}|v|^2d\th = \lim_{n\to \infty}\int_{S^1}|v_n|^2d\th = 1.$$
Thus the assumption is false and the lemma is true.
\end{proof}

In fact, we could say more about the kernel of $\pa$. In~\cite{MT}, the authors introduced the following definition. Denote the fixed point set of the whole Lie group $G$ on $M$ by $M_0$ and, for an element $\al \in \g$, denote the fixed point set of $\exp(2\pi \al)$ by $M^\al$. Obviously, $M_0$ is a subset of $M^\al$ for any $\al \in \g$.

\begin{defn}\label{d:degenerate1}
  An element $\al \in \g$ is called non-critical if $M^\al = M_0$, and critical otherwise.
\end{defn}

Observe that, for a connection $A=\al d\th$ where $\al \in \g$, if $u$ is a map such that $\pa u =0$, then $u(0)=u(2\pi)$ is a fixed point of $\exp(2\pi \al)\in G$. Then it easy to see the relationship between the above two definitions. Indeed, if $\al$ is non-critical, then any $u\in \ker(\p_{\th, \al})$ must be a constant map which lies in the fixed point set $M_0$. Thus the kernel of $\p_{\th, \al}$ is minimal and any sequence $A_n$ which converges to $A$ is non-degenerate. On the other hand, if $\al$ is critical, then there exists a map $u\in \ker(\pa)$ which is not contained in $M_0$ and $M^\al$ is strictly larger than $M_0$. Moreover, the image of $u$ is a closed orbit fixed by $\exp(2\pi \al)$. Therefore we have the following corollary, which states that the constant $C_A$ stays bounded near a non-critical connection.

\begin{cor}\label{c:uniform}
Suppose $A=\al d\th\in \A(S^1)$ is a connection and $\al\in \g$ is non-critical, then there exists an open neighborhood $\mathscr{U}_\al\subset \g$ near $\al$, such that the corresponding Poincar\'e constant given by Lemma~\ref{l:poincare} is uniformly bounded for all connections $A'=\al'd\th, \al'\in \mathscr{U}_\al$.
\end{cor}

\subsection{Exponential decay of angular energy}\label{s1}

In this section, by following a method similar to the one in \cite{LinW}(see also \cite{P,Z}), we show that for a solution $u$ to the equation (\ref{e:section}), the $\th$-direction energy
\begin{equation}\label{e5}
 \Th(t) = \int_{\{t\}\times S^1}|\pa u|^2 d\th
\end{equation}
decays exponentially as $t$ goes from either end to the center along the cylinder. We still denote $\si_A := (1/C_A)^{1/2}$ where $C_A$ is the Pincar\'e constant given by Lemma~\ref{l:poincare}.

\begin{lem}\label{l1}
Suppose $A$ is a flat connection and $u$ is a solution to equation (\ref{e3}). Let $\Th(t)$ be the energy defined by (\ref{e5}). Then
there exists a constant $\ep_A$ depending only on $\si_A$, such that if
\begin{equation}\label{e:small1}
  \norm{Du}_{C^0(\C_T)} \le \ep_A.
\end{equation}
then for any $t\in [-T, T]$,
\begin{equation}\label{ee1}
  \Th'' \ge \frac{1}{C_A}\Th - 8\int_{S^1}|f|^2 d\th.
\end{equation}
\end{lem}
\begin{proof}
Direct computation and integration by parts yield
\begin{align*}
\frac 12 \frac{d^2}{dt^2}\Th(t)
& = \frac 12\frac{d^2}{dt^2}\int_{S^1}|\pa u|^2\\
& = \int_{S^1}|\pa u_t|^2 + \int_{S^1}(\pa  u, \pa u_{tt})\\
& = \int_{S^1}|\pa  u_t|^2 - \int_{S^1}(\pa ^2 u, u_{tt}).
\end{align*}
Using equation (\ref{e3}) to substitute $u_{tt}$, we get
\begin{equation}\label{2}
\frac 12 \frac{d^2}{dt^2}\Th(t) \ge \int_{S^1}|\pa ^2 u|^2 - \int_{S^1}(\pa ^2 u, f)+ \int_{S^1}(\pa ^2 u, \Ga(u)(D u, D u)).
\end{equation}
Next we estimate the last two terms of the above equality. Applying the Young's inequality, we have
\begin{equation}\label{3}
\Abs{\int_{S^1}(\pa ^2 u, f)} \le \frac14\int_{S^1}|\pa ^2u|^2 + 4\int_{S^1}|f|^2.
\end{equation}
On the other hand, recall that
\[ \pa ^2 u = \na ^2 u-\Ga(u)(\pa u, \pa u).\]
It follows
\begin{equation}\label{22}
\begin{aligned}
\int_{S^1}(\pa ^2 u, \Ga(u)(D u, D u)) &= \int_{S^1}(\Ga(u)(\pa u, \pa u), \Ga(u)(D u, D u))\\
&\le C(\Ga)|Du|^2\int_{S^1}|\pa u|^2,
\end{aligned}
\end{equation}
where $C(\Ga)$ is a constant depending on $\Ga$. Combining  (\ref{2}), (\ref{3}) and (\ref{22}), we arrive at
\begin{equation*}
\frac 12 \frac{d^2}{dt^2}\Th(t) \ge \frac34\int_{S^1}|\pa ^2 u|^2 - C(\Ga)|Du|^2\int_{S^1}|\pa u|^2 -  4\int_{S^1}|f|^2.
\end{equation*}
Now applying the Poincar\'e inequality in Lemma~\ref{l:poincare}, we obtain
\begin{equation*}
\frac 12 \frac{d^2}{dt^2}\Th(t) \ge (\frac{3}{4}\si_A^2 - C(\Ga)|Du|^2)\int_{S^1}|\pa  u|^2 -  4\int_{S^1}|f|^2.
\end{equation*}
Thus if $\ep_A$ is small enough such that
\begin{equation*}
  C(\Ga)\ep_A^2 \le \frac14\si_A^2,
\end{equation*}
then we get the desired inequality (\ref{ee1}).
\end{proof}

The next lemma shows that $\Th(t)$ decays exponentially along the cylinder.

\begin{lem}\label{l2}
Under the hypothesis of Lemma~\ref{l1}, we have
\begin{equation}\label{e6}
|\Th(t)| \le e^{\si_A(|t|-T)}\Big(\Th(T) + \Th(-T)+ \frac8{\si_A^2}\sup_{t\in[-T, T]} \int_{\{t\}\times S^1}|f|^2d\th + Ce^{-\si_AT}\Big).
\end{equation}
\end{lem}
\begin{proof}
Define $$F(t) := 8\int_{\{t\}\times S^1}|f|^2d\th.$$
Let $\Th_0(t)$ be the solution of
\[ \Th_0'' = \si_A^2 \Th_0 - F \]
on the interval $[-T, T]$ with boundary values
\begin{equation}\label{e32}
\Th_0(-T) = \Th(-T), \quad \Th_0(T) = \Th(T).
\end{equation}
By Lemma \ref{l1}, the energy $\Th(t)$ given by (\ref{e5}) satisfies
\[ \Th'' \ge \si_A^2 \Th - F. \]
Then the maximal principle implies that $\Th(t) \le \Th_0(t)$.

It is easy to verify that the solution $\Th_0$ is explicitly given by
\[ \Th_0(t) = A(t)e^{\si_At} + B(t)e^{-\si_At}  \]
where
\[ A(t) = \frac1{2\si_A}\int_0^t F(s)e^{-\si_As}ds + C_1, \quad B(t) = \frac1{2\si_A}\int_0^{-t} F(-s)e^{-\si_As}ds + C_2 \]
and $C_1, C_2$ are constants determined by the boundary data (\ref{e32}). Obviously, the function $A(t)$ is increasing and $B(t)$ is decreasing since
$A'(t) \ge 0$ and $B'(t) \le 0$. Also note that $\Th_0(t)$ is always positive. Then for any $t$, a simple calculation yields
\[ \abs{\Th_0(t)} = A(t)e^{\si_At} + B(t)e^{-\si_At} \le (A(T) + B(-T))e^{\si_A|t|} \le e^{\si_A(|t|-T)}(A(T)e^{\si_AT}+B(-T)e^{\si_AT}). \]
By the boundary condition (\ref{e32}), we have
\begin{equation*}
|A(T)e^{\si_AT}+B(-T)e^{\si_AT}| \le \Th(T) + \Th(-T) + \abs{A(-T)e^{-\si_AT}} + \abs{B(T)e^{-\si_AT}}.
\end{equation*}
On the other hand,
\[ \abs{A(-T)e^{-\si_AT}} \le \abs{\frac1{2\si_A}\sup_{t\in[-T, 0]} F(t)\int_0^{-T}e^{-\si_A(s+T)}ds + C_1e^{-\si_AT}} \le
\frac1{2\si_A^2}\sup_{t\in[-T, 0]} F(t) + C_1e^{-\si_AT}. \]
Similarly, we have
\[ \abs{B(T)e^{-\si_AT}} \le \abs{\frac1{2\si_A}\sup_{t\in[0,T]} F(t)\int_0^{-T}e^{-\si_A(s+T)}ds + C_2e^{-\si_AT}} \le
\frac1{2\si_A^2}\sup_{t\in[0,T]} F(t) + C_2e^{-\si_AT}. \]
Combining the above inequalities, we conclude that
\begin{equation*}
|\Th(t)| \le e^{\si_A(|t|-T)}\Big(\Th(T) + \Th(-T) + \frac1{\si_A^2}\sup_{t\in[-T, T]} F(t) + (C_1+C_2)e^{-\si_AT} \Big)
\end{equation*}
and the lemma follows.
\end{proof}

Obviously, we have
\[ \Th(T) + \Th(-T) \le C\norm{Du}_{C^0(\C_T)}. \]
Integrating (\ref{e6}), we immediately obtain the following estimate of angular energy of $u$.

\begin{cor}\label{c1}
Under the hypothesis of Lemma~\ref{l1}, there is a constant $C(\si_A)$ depending on $\si_A$ such that for $T$ sufficiently large,
we have
\begin{equation*}
\int_{-T}^T |\Th(t)|dt \le C(\si_A)\Big(\norm{Du}_{C^0(\C_T)}+ \norm{f}^2_{L^\infty(\C_T)}\Big).
\end{equation*}
\end{cor}

\subsection{A gap theorem}

There is a useful conclusion which we can draw from the exponential decay of angular energy obtained above, which will be applied in the blow-up analysis below. First we introduce the following definitions.
\begin{defn}\label{d:twisted-harmonic}
A map $u$ satisfying $D^*D u = 0$ is called a twisted harmonic map (with respect to connection $A$).
\end{defn}
The twisted harmonic map is a natural generalization of classical harmonic map and has an interesting relation with the self-duality equations(cf. \cite{D}). In particular, if the connection is trivial, then we have $D=\nabla$, the Levi-Civita connection on $M$, and the equation becomes
\[ \tau(u) := -\n^*\n u = 0. \]
Thus a twisted harmonic map with respect to a trivial connection is just a harmonic map. Recall that a geodesic is a 1 dimensional harmonic map. Here we also have a parallel notion of geodesic for twisted harmonic maps.

\begin{defn}\label{d:twisted-geodesic}
Suppose $A=\al d\th$ is a flat connection on $\Real\to S^1$ and $u$ is a twisted harmonic map. If $u$ satisfies $\p_{\th, \al} u = 0$, then we say $u$ is a twisted geodesic.
\end{defn}
Obviously, a twisted geodesic satisfies the equation for geodesic
\[ \p_t u + \Ga(u)(\p_t u, \p_t u) = 0 \]
for any fixed $\th\in S^1$ since $\p_{\th, \al} u = 0$ vanishes. The image of a twisted geodesic is a closed orbit of a geodesic. Namely, for each fixed $\th\in S^1$, the curve $l_\th := u(\cdot,
\th):\Real \to M$ is a geodesic, while for each fixed $t \in \Real^1$, the curve $c_t : =u(t, \cdot):S^1\to M$ is a closed
orbit under the action of the one-parameter subgroup generated by $\al$. In other words, $u$ lies in the fixed point set $M^\al$ of $\al$. In particular, if $\al$ in non-critical, then $u$ is nothing but a geodesic lying in the fixed point set $M_0$.

It is well-know that there is a gap theorem for harmonic maps. Namely, a harmonic map with energy less than a certain small constant depending on the target manifold must be trivial~\cite{SU}. Here we show an analogous result for twisted harmonic maps.
\begin{thm}\label{t:energy-lower-bound}
Suppose $A=\al d\th$ is a flat connection and $u$ is a twisted harmonic map on the cylinder $\Real^1 \times S^1$ which satisfies $D^*D u =
0$. There exist a constant $\ep_A'$ depending on $A$ such that if
$$\sup_{t\in \Real}\norm{D u}_{L^2(\C_1(t))} \le \ep_A',$$
then $u$ is a twisted geodesic. In particular, if $\al$ is non-degenerate, then $u$ is a geodesic lying in the fixed point set $M_0$.
\end{thm}
\begin{proof}
Since $u$ is a twisted harmonic map, by the $\ep$-regularity (Lemma~\ref{l:reg2}) we have
\[ \norm{D u}_{C^0} \le C\sup_{t\in \Real}\norm{D u}_{L^2(\C_1(t))}. \]
Thus if $\ep_A'$ is small enough, then the $\norm{D u}_{C^0}$ is smaller than $\ep_A$ given in Lemma~\ref{l1}. Applying Lemma~\ref{l2} on a cylinder $\C_T=[-T,T]\times S^1$, we get
\[ |\Th(t)| \le C(\si_A)\exp(\si_A(|t|-T)) \]
for any $t\in [-T, T]$. Letting $T\to \infty$, we get $\Th(t) = 0$. It follows that $\pa  u =0$ and $u$ is a twisted geodesic which lies in the fixed point set $M^\al$.
\end{proof}

\subsection{Estimate of radial energy}

Now we are going to estimate the energy along the $t$ direction. Since we already obtain an estimate of the angular energy by Corollary~\ref{c1}, we only have to compute the difference of radical and angular energies. The idea comes from an observation in \cite{Z}.

Define a function along the cylinder by
\begin{equation}\label{e4}
e(t) := \int_{\{t\}\times S^1}(|u_t|^2 - |\pa  u|^2) d\th.
\end{equation}
Note that $e$ is gauge invariant, thus is well-defined. The next lemma shows that if there is no energy concentration, then $e(t)$ is
almost a constant.

\begin{lem}\label{l4}
Suppose $u$ is a solution to equation (\ref{e3}) and $e$ is the function defined by (\ref{e4}). Then
\begin{equation}\label{e8}
|e(t)-e(0)| \le 2\norm{Du}_{C^0(\C_T)}\cdot\norm{f}_{L^1(\C_t)}.
\end{equation}
\end{lem}
\begin{proof}
A simple calculation yields
\begin{align*}
  \frac 12 \dt{}e(t) &= \int_{S^1}(u_t, u_{tt}) - \int_{S^1}(\pa  u, \pa  u_t)\\
  &= \int_{S^1}(u_t, u_{tt} + \pa ^2u)\\
  &= \int_{S^1}(u_t, -\Ga(u)(Du, Du) + f)\\
  &= \int_{S^1}(u_t, f).
\end{align*}
It follows
\begin{equation*}
\begin{aligned}
|e(t) - e(0)| &= \int_0^t e'(t) dt= 2\int_0^t\int_{S^1}(u_t, f)d\th dt\\
&\le 2\norm{Du}_{C^0(\C_T)}\cdot\norm{f}_{L^1(\C_t)}.
\end{aligned}
\end{equation*}
\end{proof}

Integrating (\ref{e8}) on the interval $[-T, T]$, we obtain the following estimate of the energy on the $t$-direction.

\begin{cor}\label{c2}
Under the assumptions of Lemma~\ref{l4}, we have
\begin{equation}\label{e33}
\Big|\int_{-T}^T e(t)dt - 2e(0)T\Big| \le C\norm{Du}_{C^0(\C_T)}\int_{-T}^T\norm{f}_{L^1(\C_t)}dt.
\end{equation}
\end{cor}

Note that the energy of $u$ on the cylinder can be expressed as
\[ \norm{Du}_{L^2(\C_T)}^2 = 2\int_{-T}^T \Th(t)dt + \int_{-T}^Te(t)dt. \]
Combining Corollary~\ref{c1} and Corollary~\ref{c2}, we obtain the following estimate of the energy.
\begin{lem}\label{c:energy-estimate}
Suppose $u$ is a solution to equation~(\ref{e:section}) on $\C_T$ and $\ep_A$ is given by Lemma~\ref{l1}. If
\begin{equation*}
\norm{Du}_{C^0(\C_T)} \le \ep_A,
\end{equation*}
then there exists a constant $C(\si_A)$ depending on $\si_A$, such that
\begin{equation*}
|\norm{Du}_{L^2(\C_T)}^2 - 2e(0)\cdot T| \le C(\si_A)\left(\norm{Du}_{C^0(\C_T)}(1+\int_{-T}^T\norm{f}_{L^1(\C_t)}dt)
+\norm{f}_{L^\infty(\C_T)}^2\right).
\end{equation*}
\end{lem}

\section{Convergence of YMH fields away from nodes}\label{s:away}

\subsection{Convergence away from nodes}

Now we return to the main problem of this paper: the convergence of YMH fields over Riemann surfaces where the metric (or conformal structure) may degenerate.

First recall that if the metric is fixed, then the convergence of YMH fields (taking bubbles into account) is very similar to the bubble convergence of harmonic maps from surfaces to compact manifolds(Theorem~\ref{t:fixed-metric}). Because dimension 2 is sub-critical for the Yang-Mills functional, the appearance of the connection dose not make any contribution during the blow-up analysis. The proof of the bubble convergence is essentially based on the $\ep$-regularity (Lemma~\ref{l:reg1}) and removable singularity(Lemma~\ref{l:sing}).

When the metric is degenerating, we divide the convergence into two parts: away from the nodes and near the nodes. Since the metric converges away from the nodes. It is easy to see that the convergence in a compact set away from the nodes is analogous to the case of fixed metric. Since the proof is standard, here we only give an outline. For more details, we refer to our previous paper~\cite{S}. Note that one only needs to set $\al=1$ in \cite{S} to obtain Theorem~\ref{t:fixed-metric}.

Concretely, suppose $(\Si_n, h_n)$ is a sequence of Riemann surfaces which converges to a nodal surface $(\Si, h, \z)$, where $h_n$ is the canonical metric chosen in Section~\ref{s:metric}. Suppose $G$ is a compact connected Lie group and $(M, \om)$ is a compact symplectic manifold which supports a Hamiltonian action of $G$. Let $P(\Si_n)$ be the principal $G$-bundle on $\Si_n$ and $\F(\Si_n)=P(\Si_n)\times_G M$ be the associated fiber bundle. Denote by $\A(\Si_n)$ the set of smooth connections on $P(\Si_n)$ and $\S(\Si_n)$ the set of smooth sections of $\F(\Si_n)$. Suppose $\{(A_n, \phi_n)\}_{n=1}^\infty\subset
\A(\Si_n)\times\S(\Si_n)$ is a sequence of YMH fields with bounded YMH energy.

Since $h_n$ converges to $h_0$ smoothly in any compact subset $\Si'\in \Si \setminus \z$, the constant $\ep_0$ given by Lemma~\ref{l:reg1} can be chosen uniformly on $\Si'$. Let $\ep<\ep_0$ be a positive constant which we will determine later. Then the bound of the YMH energy implies that there exist at most finitely many points where the energy $\E(A_n, \phi_n)$ concentrates. Namely, there exists at most finitely many points $\x = \{x_1, x_2, \cdots, x_k\}\subset \Si'$ and geodesic balls $U(x_i, r_n)$ with radius $r_n\to 0$ such that
\[ \lim_{n\to \infty}\int_{U(x_i, r_n)}|D_n\phi_n|^2 \ge \ep. \]
Thus for any $x\in \Si'\setminus \x$, there exists a neighborhood such that Lemma~\ref{l:reg1} applies. Consequently, we can find a subsequence of $(A_n, \phi_n)$ (still denoted by $(A_n, \phi_n)$) which converges smoothly to a YMH filed $(A_\infty, \phi_\infty)$ up to gauge on $\Si'\setminus \x$. In particular, $A_\infty$ is well-defined on $\Si'$ and the convergence of $A_n$ extends over the points $\x$ (in $W^{2,p}, 1<p<2$). Therefore, by the removable singularity(Lemma~\ref{l:sing}), the limit section $\phi_\infty$ can by extended over the singular points $\x$. By taking an exhaustion, it is easy to see that the limit YMH field $(A_\infty, \phi_\infty)$ is actually defined over $\Si\setminus \z$. Moreover, the restriction of bundles $P(\Si_n)$ and $\F(\Si_n)$ on $\Si'$ gives rise to limit bundles $P(\Si\setminus \z)$ and $\F(\Si\setminus \z)$ respectively. On the other hand, one can perform a blow-up procedure(cf. Section 5 of \cite{S}) on the singular points $\x$ where the energy concentrates. Note that the bundles can always be trivialized locally. Hence near the singular points, the sections $\phi_n$ can be regarded as a sequence of maps $u_n$ into $M$. After rescaling, the connection vanishes and the rescaled maps converges to a harmonic map on $\Real^2$, which can be extended to a harmonic
sphere(bubble). Furthermore, from the neck analysis of harmonic maps(cf. \cite{P}), it can be showed that the images of the bubbles and the limit map $\phi_\infty$ are connected and the energy identity holds true. To summarize, we have the following bubble convergence theorem away from nodes.

\begin{thm}\label{t:away-from-node1}
  Suppose $\{(A_n, \phi_n)\}_{n=1}^\infty\subset \A(\Si_n)\times\S(\Si_n)$ is a sequence of YMH fields with bounded YMH energy.
  Then there is a limit fiber bundle over $\Si\setminus \z$ and a YMH field $(A_\infty, \phi_\infty)\in \A(\Si\setminus \z)\times
  \S(\Si\setminus \z)$ such that for any compact subset $\Si'\subset \Si\setminus \z$ the following hold:
  \begin{enumerate}
    \item There exist finitely many points $\mathbf{x} = \{x_1, x_2, \cdots, x_k\}\subset \Si'$ such that $(A_n, \phi_n)$
        sub-converges, up to gauge, to $(A_\infty, \phi_\infty)$ in $C^\infty_{loc}$ on $\Si'\setminus \x$.
    \item There exist finitely many harmonic spheres $w_{ij}: S^2 \to M$ where $1\le i\le k$ and $0\le j \le l_k$ such that
      \begin{equation}
          \lim_{n\to \infty} \YMH_{h_n}(A_n, \phi_n)|_{\Si'} = \YMH_h(A_\infty, \phi_\infty)|_{\Si'} + \sum_{i,j}\E(w_{ij}),
      \end{equation}
      where $\E(w_{ij}) = \norm{dw_{ij}}_{L^2}$ is the energy of $w_{ij}$. Moreover, the images of the tree bubbles $w_{ij}$ and
      the limit section $\phi_\infty$ are connected.
  \end{enumerate}
\end{thm}

Therefore, to prove Theorem~\ref{t:main1}, it remains to investigate the asymptotic behavior of the limit YMH field $(A_\infty,
\phi_\infty)$ at the nodes.

\subsection{Limit holonomy on punctured disk}

In general, the limit fiber bundle and the pair $(A_\infty, \phi_\infty)$ obtained in Theorem~\ref{t:away-from-node1} can not be
extended over the nodes to the whole surface $\Si$. The obstruction can be seen as follows.

Let $\D$ be the unit flat disk centered at $0$ with radius $1$ and $\D^* = \D\setminus\{0\}$ the punctured disk. Suppose $P$ is a
principal $G$-bundle on the punctured disk $\D^*$, $A$ is a connection and $F$ its curvature. Since the punctured disk $\D^*$ is not
simply connected, the bundle $P$ might be non-trivial in general. In fact, the bundle can only be extended to the origin if the holonomy
of the connection around the origin is trivial. Nevertheless, the holonomy do has a limit if $F$ belongs to $L^p$ for some $p>1$.

More precisely, under the polar coordinates $(r, \th)$ of $\D$, let $l_\th := \{(r, \th)| 0<r<1\}$ be the line of angle $\th$ and $c_r =
\{x\in \D^*||x| = r\}$ be the circle with radius $r>0$. Chose an orthogonal normal frame $\{e_i(r, \th)\}$ along the line $l_0$ and extend
them by parallel translation around the circle $c_r$ for every $0<r<1$. Then $D_{\th, A} e_i = 0$ and holonomy appears. Suppose
$e_i(r, 2\pi) = e_i(r, 0)\cdot g(r)$ for some $g(r) \in G$, then we define the holonomy on $c_r$ by
\[ \Hol(A, r) = [g(r)], \]
where $[g(r)]$ denotes the conjugacy class of $g(r)$. In \cite{S}, the author proved the following theorem.
\begin{thm}\label{t:holonomy}
If $\norm{F}_{L^p} \le C$ for some $p>1$, then their exists $g_0 \in G$ such that
\[ \Hol(A, 0) := \lim_{r\to 0} \Hol(A, r) = [g_0]. \]
\end{thm}
\begin{rem}
It's obvious that if $g_0 = id$, then the bundle is trivial and can be extended to the whole disk $\D$.
\end{rem}

Now suppose $z\in \z$ is a node and $z_1, z_2$ are the preimages of $z$. We may take two punctured disks at $z_1$ and $z_2$.
Since the curvature of $A_\infty$ is $L^2$ bounded, it follows from Theorem~\ref{t:holonomy} that there exist a limit holonomy at
$z_1$ and $z_2$ respectively. It is easy to see that these two limit holonomies coincide, since they both equal to the limit of the
holonomy of $A_n$ around the curve in $\Si_n$ which shrinks to the node $z$. Obviously, $A_\infty$ can be extended over $z$ if and
only if the holonomy is trivial. In this case, the limit bundles $P(\Si\setminus \z)$ and $\F(\Si\setminus \z)$ can be extended over $z$
and the singularity of $\phi_\infty$ at $z$ can be removed by Lemma~\ref{l:sing}. Thus, we obtain the following removable singularity
on the nodes.

\begin{thm}
  Let $(A_\infty, \phi_\infty)$ be the limit YMH field on $\Si\setminus \z$ given by Theorem~\ref{t:away-from-node1}, then the limit
  holonomy $\Hol(A_\infty, z)$ on each node $z\in \z$ exists. In particular, if $\Hol(A_\infty, z) = id$, then the singularity at $z$ can be
  removed.
\end{thm}

\subsection{Limit pair at nodes}

Let $(A_\infty, \phi_\infty)$ be the limit pair obtained in Theorem~\ref{t:away-from-node1}, which is a well-defined YMH field on the
surface $\Si\setminus \z$ except the nodal points $\z$. Now we are going to investigate the asymptotic behavior of $\phi_\infty$ at the
nodes.

Let $z\in \z$ be a node and $\pi:\t{\Si}\to \Si$ be the normalization map. We choose one of the preimages of $z$, say $z_1\in \t{\Si}$.
Recall that the metric we choose is flat near $z$. Thus we may find a small $r_0>0$ such that the neighborhood $U(z_1, r) =  \{x\in
\t{\Si}| |x-z_1| \le r_0\}$ is a flat disk. Since there are only finitely many points, where the energy concentrates, we may assume that there is no energy concentrating points in $U(z_1, r)\setminus\{z_1\}$. The punctured disk $U(z_1, r)\setminus\{z_1\}$ is isomorphic to a half cylinder $\C_+ = [0, \infty)\times S^1$ endowed with metric $ds^2 = \ld^2(dt^2+d\th^2)$, where $\ld$ decays exponentially as $t$ goes to infinity. Thus we may regard the restriction of fiber bundles $P(\Si\setminus \z)$ and $\F(\Si\setminus \z)$ on $U(z_1, r)\setminus\{z_1\}$ as fiber bundles over the cylinder $\C_+$. For any $t>0$, denote the restriction of $P$ and $F$ on the circle $\{t\}\times S^1$ by $P(t)$ and $\F(t)$ respectively. Using parallel transportation, we may naturally identify the bundles $P(t)$ and $\F(t)$ with $P(0)$ and $\F(0)$ for all $t$. Also, if we denote by $(A(t), \phi(t))$ the restriction of the pair $(A_\infty, \phi_\infty)$ on $\{t\}\times S^1$, we may regard $(A(t), \phi(t))$ as a pair on $P(0)\times \F(0)$. Moreover, since the Lie group $G$ is connected and the cylinder $\C_+$ is homotopic to $S^1$, the fiber bundles can be trivialized and we identify the section $\phi(t)$ with a map $u(t):S^1\to M$. Then we have

\begin{thm}\label{t:limit-pair-at-node}
Let $(A(t), u(t))$ be the pair defined as above, then $(A(t),u(t))$ converges a pair $(A_z, u_z)$ in $C^0(S^1)$ as $t\to
\infty$. Moreover, they satisfies
  \[D_{A_\infty}u_\infty = 0. \]
\end{thm}
\begin{proof}
By changing the metric on $\C_+$ to the standard metric $ds^2_0 = dt^2+d\th^2$, it is obvious that the $L^2$ norm of curvature on $[t, +\infty]\times S^1$ decays exponentially as $t\to \infty$. In view of Lemma~\ref{l5}, it's easy to see that $A(t)$ converges, modula gauge, to a flat connection $A_\infty = \al d\th$ defined over $S^1$.

On the other hand, there exist a constant $\ep>0$ such that
\[ \norm{D_{A_\infty}u_\infty}_{L^2(\C_1(t))} < \ep \]
for all $t$. Otherwise we get a contradiction to the assumption that there are no energy concentrating points on $\C_+$. In particular, we may set $\ep< \ep'_{A_\infty}$, where $\ep'_{A_\infty}$ is given by Theorem~\ref{t:gap}. Then by an argument similar to the proof of Theorem~\ref{t:gap}, we get the exponential decay
\begin{equation}\label{1}
  \int_{S^1} |D_{A_\infty} u(t)|^2 d\th \le C(\si_{A_\infty})\exp\si_{A_\infty}(T-t).
\end{equation}
Hence the convergence of $u(t)$ to a map $u_\infty:S^1\to M$ in $C^0$ follows. Moreover, by letting $t\to \infty$ in (\ref{1}), we find that the limit pair satisfies the equation
\[ D_{A_\infty}u_\infty = 0. \]
\end{proof}


\section{Convergence near nodes}\label{s:near}

\subsection{From collar to cylinder}\label{s:reduced-problem}

In this section, we focus on the blow-up analysis near the nodes, which is the most interesting part. Again we suppose $(\Si_n, h_n, j_n)$ is a sequence of Riemann surfaces which converges to a nodal surface $(\Si, h, j)$ with nadal set $\z$, where $h_n$ is the canonical metric chosen in Section~\ref{s:metric} which is flat near the nodes. Then for any node $z\in \z$, there is a collar area in $\Si_n$ which is isomorphic to a long cylinder $\C_n = [-T_n, T_n]\times S^1$, endowed with a metric $g_n = \ld_n^2(dt^2+d\th^2)$. Moreover, $\ld_n$ is exponentially bounded by
\begin{equation}\label{e:exp-decay}
  |\ld_n(t)| \le C\de_n\exp(|t| - T_n), \quad \forall t\in [-T_n, T_n],
\end{equation}
where $\de_n = e^{-T_n}\to 0$ as $n\to \infty$. After restricting the bundles and YMH fields to the collar area, we obtain a sequence of
bundles $P(\C_n)$ and $\F(\C_n)$ on cylinders $\C_n$, as well as a sequence of YMH fields $\{(A_n,
\phi_n)\}_{n=1}^\infty\subset\A(\C_n)\times\S(\C_n)$ with bounded YMH energy with respect to the metric $g_n$. Namely, $(A_n, \phi_n)$ satisfies the Euler-Lagrangian equation~(\ref{e:el1}) and there exists a constant $C$
such that
\begin{equation}\label{e000}
  \YMH_{g_n}(A_n, \phi_n) =\norm{D_n \phi_n}_{L^2,g_n}^2 + \norm{F_n}_{L^2,g_n}^2 + \norm{\mu(\phi_n) - c}_{L^2,g_n}^2
  \le C.
\end{equation}
Note that by assumption the bundles on the cylinder are trivial and we always identify a section $\phi_n$ over $\C_n$ with a map
$u_n:\C_n \to M$.

To study the convergence of $\{(A_n, u_n)\}_{n=1}^\infty$ on the cylinder, we need to restate the problem under the
standard flat metric $g_0 = dt^2 + d\th^2$. First note that $g_n=\ld_n^2g_0$ is conformal to $g_0$. By the conformal property of the
YMH functional~(\ref{e:conformal}), the energy bound~(\ref{e000}) implies that for any sub-cylinder $\C_t := [-t, t]\times S^1$
\begin{equation*}
  \norm{D_n u_n}^2_{L^2(\C_t),g_0} + (\sup_{t\in \C_t}\ld)^{-2}\norm{F_n}^2_{L^2(\C_t),g_0} \le  C.
\end{equation*}
It follows from (\ref{e:exp-decay}) that for the conformal metric $g_0$, we have
\begin{equation}\label{e:rescale}
\norm{D_{n} u_n}_{L^2, g_0} \le C,  \ \  \norm{F_{n}}_{L^2(\C_t),g_0} \le C\de_n\exp(|t| - T_n).
\end{equation}
Moreover, after the conformal transformation, the Euler-Lagrangian equation~(\ref{e:el1}) becomes
\begin{equation}\label{e:el2}
  \left\{
  \begin{aligned}
  D^*_n D_n u_n &= -\ld_n\nabla\H(u_n),\\
  D^*_n F_n &= -\ld_n u_n^* D_n u_n,
  \end{aligned}
  \right.
\end{equation}
where $D_n^*$ denotes the dual of $D_n$ under the metric $g_0$. Then we are concerned with the convergence behavior of the
sequence $(A_n, u_n)$ over the standard cylinder $(\C_n, g_0)$ which satisfies (\ref{e:rescale}) and (\ref{e:el2}).

\subsection{Energy concentrations on cylinder}\label{s2}

By the $\ep$-regularity (Lemma~\ref{l:reg1}), it is easy to see that the connection $A_n$ converges locally in $C^0$ to a limit connection $A_\infty$ which
is defined over the infinitely long cylinder $\C_\infty := \Real^1\times S^1$. Moreover, by (\ref{e:rescale}), we have
\[ \lim_{n\to \infty} \norm{F_{n}}_{L^2(\C_t),g_0} = 0 \]
for any sub-cylinder $\C_t$ with fixed length. It follows that the limit connection $A_\infty$ is flat. On the other hand, to apply the $\ep$-regularity on the map $u_n$, we need the
assumption of smallness of the energy. That is, if $\norm{D_{n} u_n}_{L^2(\D_x)}<\ep_0$ on a disk $\D_x$ centered at $x$, then
$u_n$ converges strongly near $x$. However, because $\norm{D_{n} u_n}_{L^2}$ is conformally invariant,  the energy may
concentrate and the blow-up phenomenon happens. We distinguish two cases, where two kinds of bubbles emerge.

The first case is energy concentration near a point. Namely, there exists $x_n \in \C_n$, $\de_n>0$ and a positive constant $\ep>0$ such that $x_n \to x_0$,
$\de_n\to 0$ and
\[ \lim_{n\to \infty} \norm{D_n u_n}_{L^2(B_{x_n}(\de_n))} \ge \ep \]
where $B_{x_n}(\de_n)$ is a disk of radius $\de_n$ centered at $x_n$. In this case, by a standard rescaling procedure and the removable singularity theorem (Lemma~\ref{l:sing}), the
connection vanishes and one obtains finitely many bubbles which are just harmonic spheres as in Theorem~\ref{t:fixed-metric}. More precisely, we define the rescaled pairs
\[ \tilde{A}_n(y) = A(x_n + k_n y), ~~ \tilde{u}_n(y) = u_n(x_n + k_n y), y\in B_0(R)\]
where $\lim_{n\to \infty}k_n = 0$ and $B_0(R)$ is a disk centered at the origin on $\Real^2$ with radius $R$. Then the curvature $F_{\tilde{A}_n}$ vanishes as $n\to \infty$ and the connection $\tilde{A}_n$ converges to a trivial connection. Moreover, $\tilde{u}_n$ converges to a harmonic map defined on $B_0(R)$. By letting $R$ go to infinity, we obtain a harmonic map on $\Real^2$, which can be extended to a bubble, i.e. a harmonic map from $S^2$. Since the energy of harmonic spheres has a lower bound, there are at most finitely many such bubbles(\cite{SU}). In particular, by choosing $k_n$ carefully, we will get a set of bubbles and the image of these bubbles are connected(see \cite{P} for a detailed construction of the bubble tree). We call this kind of bubbles the \emph{tree bubbles}. (See \cite{S} for more details.)

The second case is energy concentration on a "drifting" sub-cylinder. Namely, there exists $t_n\in [-T_n, T_n]$ such that
\begin{equation}\label{e:tbubble}
  \lim_{n\to \infty} \norm{D_n u_n}_{L^2(\C_1(t_n))} \ge \ep,
\end{equation}
where $\C_1(t_n) = [t_n - 1, t_n + 1] \subset \C_n$ is a unit length sub-cylinder centered at $t_n$. Here we assume that there is no
energy concentrations near a point as in the first case, i.e. we have
\[ \lim_{\de\to 0}\lim_{n\to \infty} \norm{D_n u_n}_{L^2(B_{x}(\de))} = 0, ~~\forall x \in \C. \]
In this case, we first shift the origin by $t_n$ and consider the map
\[  \tilde{u}_n(t,\th) := u_n(t+t_n, \th). \]
By possibly extending the cylinder $\C_n$ to a longer one, which corresponds to a slightly larger collar area in the Riemann surface, we may regard $\tilde{u}_n$
as a map defined on $[-T_n', T_n']\times S^1$ for some $T_n'\to \infty$. Then it follows by the $\ep$-regularity that $(u_n, A_n)$
converges strongly in $W^{1,2}$ to a pair $(u_\infty, A_\infty)$ on the infinite cylinder $\C_\infty =
\Real^1\times S^1$ which satisfies
\begin{equation}\label{e:twisted-bubble}
D_{A_\infty}^*D_{A_\infty} u_\infty = 0.
\end{equation}
That is, $u_\infty$ is a twisted harmonic map with respect to the flat connection $A_\infty$. We call such a limit pair a \emph{twisted bubble}(or \emph{connecting bubble}).

Note that if the connection $A_\infty$ is trivial, the twisted harmonic map is indeed a classical harmonic map $u_\infty$ defined on
$\Real^1\times S^1$ and again can be extended to a harmonic sphere defined on $S^2$. But in general, the limit connection $A_\infty$,
although flat, may have a non-trivial holonomy on the cylinder. As shown before, in contrary to the tree bubbles, the twisted bubbles do
have point singularities which can not be extended over. Another difference from the tree bubbles is that the energy of a twisted bubble
can be arbitrarily small as the limiting connection $A_\infty$ varies. Fortunately, for a fixed connection, we do have the following gap theorem.
\begin{thm}\label{t:gap}
   There exist a constant $\ep'_{A_\infty}$ depending on $A_\infty$ such that any twisted bubble $u_\infty$ with energy smaller than
   $\ep'_{A_\infty}$ is trivial, i.e. a constant map.
\end{thm}
\begin{proof}
Since $A_\infty$ is flat, by Theorem~\ref{t:energy-lower-bound}, there exists $\ep'_{A_\infty}>0$ depending on $A_\infty$, such that
any twisted harmonic map on $\C_\infty$ with energy smaller than $\ep'_{A_\infty}$ is a twisted geodesic. Thus the twisted bubble $u_\infty$ satisfies
$\p_{\th, A_\infty}u_\infty=0$ and the equation~(\ref{e:twisted-bubble}) reduces to
\[\p_t u_{\infty} +\Ga(u_\infty)(\p_t u_\infty, \p_t u_\infty) = 0.\]
Therefore, for any $\th \in S^1$, the curve $\ga_\th(t) : = u_\infty(t, \th)$ is a geodesic. It follows that $|\p_t \ga_\th|$ is constant and the
energy of $u_\infty$ on a sub-cylinder of unit length is constant. Namely, for any $t$ and sub-cylinder $\C_1(t)$, we have
\[ \norm{D_\infty u_\infty}_{L^2(\C_1(t))} = \Big(\int_{t-1}^{t+1}\int_{S^1}|\p_t u_\infty|^2 d\th dt\Big)^{1/2} =
\Big(2\int_{S^1}|\p_t \ga_\th|^2d\th\Big)^{1/2}= const. \]
On the other hand, from (\ref{e:tbubble}) we have
$$  \norm{D_\infty u_\infty}_{L^2(\C_1(t))} = \lim_{n\to \infty}\norm{D_n u_n}_{L^2(\C_1(t_n))} \ge \ep, $$
where $\ep$ is a positive number. This implies that the limit map $u_\infty$ would have infinite energy on the infinite cylinder $\C_\infty$. However, this
contradicts with the assumption of boundedness of energy. This proves the theorem.
\end{proof}

An immediate application of the above theorem is the finiteness of number of bubbles.
\begin{thm}
There are at most finitely many bubbles (including tree bubbles and twisted bubbles) appearing during the blow-up process.
\end{thm}
\begin{proof}
  First recall that the energy of harmonic spheres are bounded from below by a positive constant. It follows that there are at most finitely many tree bubbles. As for the twisted bubbles, note that there are at most finitely many nodes. At each node $z\in \z$ there is a unique limit flat connection $A_\infty$ and a corresponding constant $\ep_z:= \ep'_{A_\infty}$ given by Theorem~\ref{t:gap}. Thus there is a positive minimum $\ep_1=\min_{z\in\z}\ep_z$ such that any twisted bubble contains a finite amount of energy larger than $\ep_1$. Since the total YMH energy is bounded, there are at most finitely many twisted bubbles as well.
\end{proof}

Therefore, after finitely many steps of blowing-up's, we may assume that there is no energy concentration(cf. \cite{DT}). More precisely, we say that there is
\emph{no energy concentration} on the cylinder if
\begin{equation}\label{e:no-concentration}
  \lim_{n\to \infty}\sup_{t\in[-T_n, T_n]}\norm{D_n u_n}_{L^2(\C_1(t))} = 0.
\end{equation}

\subsection{Energy identity}

By the arguments in last subsection, we may assume that there is no energy concentrations on the cylinder. Therefore there is no energy on any sub-cylinder with fixed-length. However, since the length of the whole cylinder tends to infinity, the neck might still contain a positive amount of energy. An important issue then is to compute the accumulated energy on the cylinder.

To do this, we first put the connection $A_n$ in balanced temporal gauge such that $A_n = a_n d\th$ where $a_n:\C_n\to \g$ and the
restriction of $a_n$ on the middle circle $\{0\}\times S^1$ is a constant $\al_n \in \g$. By Remark~\ref{r1}, we can assume that $\al_n$ converges to some $\al_\infty \in \g$. Then we denote by $\bar{A}_n = \al_n d\th$ the corresponding flat connection and let $\pn := \p_\th + \al_n$ be the partial differential operator induced by $\bar{A}_n$. Next define
\begin{equation}\label{e46}
e_n(t) := \int_{\{t\}\times S^1}(|\p_tu_n|^2-|\pn u_n|^2)d\th,
\end{equation}
and let $e_n := e_n(0)$. Moreover, we denote
\begin{equation}\label{e45}
\mu := \lim_{n\to \infty}T_n e_n
\end{equation}
and the energy of $u_n$ on $\C_{n}$ by
$$ \E(u_n, A_n, \C_n) = \int_{\C_n}|D_n u_n|^2 dtd\th. $$
It is obvious that by choosing a sub-sequence, the limit $0\le \mu <\infty$ exists. Our main result of this section is the following generalized energy identity.

\begin{thm}[Energy identity]\label{t:energy-identity}
Suppose $\{(A_n, u_n)\}_{n=1}^\infty\subset \A(\C_n)\times\S(\C_n))$ is a sequence of YMH fields on cylinder $(\C_n, g_0)$ which
satisfies the energy bound (\ref{e:rescale}) and the rescaled EL equation~(\ref{e:el2}). Then there exists a subsequence, which we still denote by $(A_n, u_n)$, such that $A_n$ converges smoothly to a flat connection $A_\infty$ on $\C_\infty$. Moreover, if there is no energy concentration, then the limit of $u_n(\C_n)$ falls into the fixed point set of $\exp(2\pi \al_\infty)$ and the following hold.
\begin{enumerate}
\item If $A_n$ is non-degenerating, then the energy on the neck is
   \[ \lim_{n\to \infty}\E(u_n, A_n, \C_n) = 2\mu. \]
\item If $A_n$ is degenerating, then
   \begin{equation}\label{e13}
   \lim_{n\to \infty}\E(u_n, A_n, \C_n) = 2\lim_{n\to \infty}\int_{\C_n}|(\al_n-\al_\infty)u_n|^2dtd\th + 2\mu.
   \end{equation}
 \end{enumerate}
\end{thm}
\begin{proof}
First observe that there is no energy on the unit-length sub-cylinders $[-T_n, -T_n +1]\times S^1$ and $[T_n-1, T_n]\times S^1$ at the ends of the cylinder $\C_n$. Thus for convenience, we may set $T_n = T_n-1$ and $\C_n = [-T_n+1, T_n-1]\times S^1$.

For any $\ep>0$ which is smaller than the $\ep$'s appeared before, the assumption of no energy concentration (\ref{e:no-concentration}) and the $\ep$-regularity implies that
\begin{equation}\label{e40}
  \norm{D_n u_n}_{C^0(\C_n)} <\ep
\end{equation}
for sufficiently large $n$.

By Lemma~\ref{l:gauge}, we can put the connection $A_n$ in balanced temporal gauge such that $A_n = a_n d\th$ where $a_n(0, \th) = \al_n \in \g$ for all $\th \in S^1$. The equation for $A_n$ in~(\ref{e:el2}) can be written as
\begin{equation}\label{e41}
 D_n^*F_n = \ld_n B_n
\end{equation}
where $B_n = u_n^*D_nu_n$ and $\ld_n$ is exponentially bounded by (\ref{e:exp-decay}). By (\ref{e40}), we have
$\norm{B_n}_{C^0(\C_n)} < \ep$. Thus equation (\ref{e41}) has the form of (\ref{e7-1}) and satisfies the hypothesis
of Lemma~\ref{l5}. If follows that
\begin{equation}\label{e42}
  \norm{a_n-\al_n}_{W^{1,\infty}(\C_{t})} \le C\de_n\exp(|t| - T_n), \forall t\in [-T_n, T_n]\times S^1.
\end{equation}
Recall that we assume that $\al_n$ converges to some $\al_\infty \in \g$ by Remark~\ref{r1}. Then (\ref{e42}) implies that the connection
$A_n$ converges to a flat connection $A_\infty = \al_\infty d\th$ along the cylinder.

Combining (\ref{e40}) and (\ref{e42}), it is clear that in any fixed sub-cylinder of finite length, the map $u_n$ converges strongly to a map $u_\infty$ which satisfies
\[ D_{A_\infty}u_\infty = \lim_{n\to \infty}D_n u_n=0. \]
It follows that
$$\p_{\th, \al_\infty}u_\infty = \p_\th u_\infty + \al_\infty\cdot u_\infty = 0, $$
which implies that $u_\infty$ lies in the fixed point set $M^{\al_\infty}$. On the other hand, we also have $\p_t u_\infty=0$. Thus $u_\infty$ is independent on $t$ and the image of $u_\infty$ on a fixed sub-cylinder is just a single closed orbit. However, since the length of the cylinder $\C_n$ tends to infinity, the limit of the images $u_n(\C_n)$ does not necessarily shrink to an orbit.

In the balanced temporal gauge, the equation for $u_n$ in (\ref{e:el2}) is
\begin{equation*}
  \tau(u_n) + \p_\th a_n\cdot u_n + 2a_n\cdot \p_\th u_n + a_n^2\cdot u_n = \ld_n \nabla \H(u_n).
\end{equation*}
Using the flat connection $\bar{A}_n  = \al_n d\th$, we may rewrite the equation as
\begin{equation*}
\tau(u_n) + 2\al_n\cdot \p_\th u_n + \al_n^2\cdot u_n  = f_n,
\end{equation*}
or equivalently,
\begin{equation}\label{e12}
\bd_n^*\bd_n u_n = f_n
\end{equation}
where $\bd_n := d +\bar{A}_n$ and
\[ f_n =  \p_\th a_n\cdot u_n - 2a_n\cdot \p_\th u_n - a_n^2\cdot u_n + 2\al_n\cdot \p_\th u_n + \al_n^2\cdot u_n + \ld_n \nabla
\H(u_n). \]
In view of (\ref{e40}) and estimate of the connection (\ref{e42}), it is obvious that $f_n$ is $L^\infty$-exponentially bounded by
\begin{equation}\label{e10}
  \norm{f_n}_{L^\infty(\C_t)} \le C\de_n\exp(|t| - T_n).
\end{equation}
It follows
\[ \int_{-T_n}^{T_n}\norm{f_n}_{L^1(\C_t)}dt \le C\de_n\int_{-T_n}^{T_n}\exp(|t| -T_n)dt \le C\de_n. \]

Now we distinguish two cases.

\emph{Case 1. $A_n$ is non-degenerating.}

In this case, the equation of $u_n$ has the form of (\ref{e:section}) and by Lemma~\ref{l2}, we have exponential decay of the angular energy
\begin{equation*}
\begin{aligned}
\Th_n(t)&:= \int_{S^1\times\{t\}}|\p_{\th, \al_n}u_n|^2d\th \\
&\le e^{\si_n(|t|-T_n)}\left(\Th(T_n) + \Th(-T_n)+ \frac8{\si_n^2}\norm{f}_{L^\infty}^2 + Ce^{-\si_nT_n}\right)\\
&\le \frac{C}{\si_n^2}e^{\si_n(|t|-T_n)}(\ep + \de_n).
\end{aligned}
\end{equation*}
Here $\si_n =(1/C_{\bar{A}_n})^{1/2}$ corresponds to the Poincar\'e constant of the flat connection $\bar{A}_n$ given by Lemma~\ref{l:poincare}. Since $A_\infty$ is non-degenerate, by Lemma~\ref{l:uniform}, $\si_n$ is bounded away from zero uniformly. Therefore, we have uniform exponential decay of $\Th_n(t)$. Moreover, all the hypothesis of Lemma~\ref{c:energy-estimate} are satisfied,
which yields a uniform energy estimate
\begin{equation}\label{e44}
|\norm{\bd_nu_n}_{L^2(\C_n)}^2 - 2e_nT_n| \le C(\ep + \de_n),
\end{equation}
where $e_n$ is defined by (\ref{e46}). Note again that since the connection $A_\infty$ is non-degenerate, the constant $C$ in (\ref{e44}) can be chosen uniformly.

Consequently, by letting $n\to \infty$ in (\ref{e44}), we get
\begin{equation*}
\lim_{n\to\infty}|\norm{\bd_nu_n}_{L^2(\C_n)}^2 - 2\mu| \le C\ep.
\end{equation*}
Since $\ep$ can be taken arbitrarily small, it follows
\begin{equation*}
\lim_{n\to\infty}\norm{\bd_nu_n}_{L^2(\C_n)}^2 = 2\mu.
\end{equation*}
Finally, recall that the connection $D_n$ is close to $\bd_n$ by (\ref{e42}), we conclude that
\begin{equation*}
\lim_{n\to\infty}\norm{D_nu_n}_{L^2(\C_n)}^2 =\lim_{n\to\infty}\norm{\bd_nu_n}_{L^2(\C_n)}^2 = 2\mu.
\end{equation*}

\emph{Case 2. $A_n$ is degenerating.}

In this case, we no longer have uniform exponential decay for the angular energy $\Th_n(t)$. However, observe that by (\ref{e42}) and the convergence of $\al_n\to\al_\infty$, we have for sufficiently large $n$
\begin{equation}\label{e43}
\norm{a_n-\al_\infty}_{W^{1,\infty}} \le \norm{a_n-\al_n}_{W^{1,\infty}}+|\al_n-\al_\infty|\le C\de_n
\end{equation}
Thus we can replace the connection $\bar{A}_n$ by the limit connection $A_\infty$ in~(\ref{e12}) and rewrite the equation of $u_n$ as
\begin{equation}\label{e47}
D_{A_\infty}^*D_{A_\infty}u_n = \tilde{f}_n,
\end{equation}
where $D_{A_\infty}=\nabla+\al_\infty d\th$ and
\[ \tilde{f}_n =  \p_\th a_n\cdot u_n - 2a_n\cdot \p_\th u_n - a_n^2\cdot u_n + 2\al_\infty\cdot \p_\th u_n + \al_\infty^2\cdot u_n + \ld_n \nabla
\H(u_n). \]
Although the function $\tilde{f}_n$ no longer decays exponentially, we still have, in view of (\ref{e43}),
\begin{equation*}
|\tilde{f}_n|\le C(|\p_\th a_n|+|a_n-\al_\infty||\p_{\th, \al_n}u_n|+|\al_\infty||\p_{\th, \al_n}u_n-\p_{\th, \al_\infty}u_n|)\le C\de_n.
\end{equation*}
Therefore, we can apply Lemma~\ref{l2} for equation~(\ref{e47}) to get
\begin{equation}\label{e63}
\tilde{\Th}_n(t) := \int_{S^1\times\{t\}}|\p_{\th,\al_\infty}u_n|^2d\th\le \frac{C}{\si_\infty^2}e^{\si_\infty(|t|-T_n)}(\ep + \de_n),
\end{equation}
where $\si_\infty:=(1/C_{A_\infty})^{1/2}$ corresponds to the Poincar\'e constant of $A_\infty$. Integrating the above inequality yields
\begin{equation}\label{e48}
\int_{-T_n}^{T_n}\tilde{\Th}_n(t)dt \le \frac{C}{\si_\infty^3}(\ep+\de_n).
\end{equation}
On the other hand, the actual angular energy $\Th_n$ can be expressed as
\begin{equation*}
\begin{aligned}
\Th_n(t) &= \int_{\{t\}\times{S^1}}|\pn u_n|^2d\th\\
& =\int_{\{t\}\times{S^1}}|\p_{\th,\al_\infty} u_n + (\al_n-\al_\infty)u_n|^2d\th\\
&=\tilde{\Th}_n(t) + \int_{\{t\}\times{S^1}}\<\p_{\th,\al_\infty} u_n, (\al_n-\al_\infty)u_n\>d\th +\int_{\{t\}\times{S^1}}|(\al_n-\al_\infty)u_n|^2d\th .
\end{aligned}
\end{equation*}
Integrating and using (\ref{e48}), we get
\begin{equation}\label{e49}
\lim_{n\to \infty}\int_{-T_n}^{T_n}\Th_n(t)dt = \lim_{n\to \infty}\int_{-T_n}^{T_n}\int_0^{2\pi}|(\al_n-\al_\infty)u_n|^2d\th dt.
\end{equation}
Moreover, Lemma~\ref{l4} and Corollary~\ref{c2} still applies, yielding
\begin{equation*}
|\int_{-T_n}^{T_n} e_n(t)dt - 2e_n T_n|
\le C\norm{D_n u_n}_{C^0(\C_{T_n})}\int_{-T_n}^{T_n}\norm{f_n}_{L^1(\C_t)}dt
\le C\ep\de_n.
\end{equation*}
It follows that
\begin{equation}\label{e50}
 \lim_{n\to\infty}|\int_{-T_n}^{T_n} e_n(t)dt| = \lim_{n\to\infty}2e_nT_n = 2\mu.
\end{equation}
Combining (\ref{e49}) and (\ref{e50}), we obtain
\begin{equation*}
\begin{aligned}
\lim_{n\to \infty}\norm{D_nu_n}_{L^2(\C_n)}^2 &= \lim_{n\to\infty}\(\int_{-T_n}^{T_n} e_n(t)dt+2\int_{-T_n}^{T_n}\Th_n(t)dt\)\\
 &= 2\mu + 2\lim_{n\to \infty}\int_{\C_n}|(\al_n-\al_\infty)u_n|^2d\th dt.
\end{aligned}
\end{equation*}
\end{proof}

As a corollary, we have
\begin{cor}
  If $A_n$ is non-degenerating, the neck contains no energy if and only if $\mu = 0$.
\end{cor}

\subsection{Further analysis of the neck}

\subsubsection{The non-degenerate case}\label{s3}

For a sequence of harmonic maps with bounded energy on cylinders whose length tends to infinity, Chen, Li and Wang \cite{CLW} showed
that there exists a subsequence which converges to a geodesic. The length of the limit geodesic can be zero, finite or infinite. Here we
follow \cite{CLW} closely to investigate the geometric properties of the neck. First we assume that the connection is non-degenerating. The advantage of the non-degeneracy is that the Poincar\'e constant is bounded by Lemma~\ref{l:uniform}. Hence all the constants depending on $A_n$ can be chosen uniformly.

Again we suppose there is no energy concentration on the cylinder. Also for convenience, we set $T_n = T_n-1$ and $\C_n = [-T_n+1, T_n-1]\times S^1$, since the unit-length sub-cylinders $[-T_n, -T_n +1]\times S^1$ and $[T_n-1, T_n]\times S^1$ at the ends of the cylinder $\C_n$ does not affect the results in this section. From the analysis before, we already know that $u_n(\C_n)$ converges to a closed orbit of a curve $\ga$, which we refer as the neck. We define the length of the neck to be the length of the curve $\ga$.

Let $\bar{A}_n = \al_n d\th$ be the flat connection corresponding to $A_n$ and $\bd_n = \nabla+\bar{A}_n$, $\pn = \p_\th + \al_n$ be the derivatives induced by $\bar{A}_n$. Let $e_n$ be defined by (\ref{e46}) and
\[ \nu := \lim_{n\to \infty} T_n\sqrt{e_n}. \]
Note that $\nu$ could be $0$, finite or infinite.

\begin{lem}\label{l7}
If $\nu=0$, then the length of the neck is zero, i.e. the neck is a single closed orbit of a point.
\end{lem}
\begin{proof}
First recall that the connection $A_n$ converges to the limit flat connection $A_\infty$ by (\ref{e42}) and $u_n$ satisfies
equation~(\ref{e12}). Let $x=(t, \th)$ be a point on the cylinder and $\C_1(t) = [t-1, t+1]\times S^1$ be a sub-cylinder. Since there is no energy concentration, by the $\ep$-regularity (Lemma~\ref{l:reg2}), we have
\begin{equation}\label{e531}
|\bd_n u_n(x)| \le C(\norm{\bd_n u_n}_{L^2(\C_1(t))} + \de_n\exp\si(|t| - T_n)),
\end{equation}
where $\si$ is a constant independent of $n$.
Moreover, by Lemma~\ref{l2}, we have exponential decay for the $\th$-direction energy
\begin{equation}\label{e53}
\norm{\pn u_n}_{L^2(\C_1(t))}^2 \le C(\ep+\de_n)\exp\si(|t| - T_n).
\end{equation}
For the $t$-direction, by Lemma~\ref{l4}, we have
\begin{equation}\label{e532}
\norm{\p_t u_n}_{L^2(\C_1(t))}^2 \le 2e_n + C(\ep+\de_n)\exp\si(|t| - T_n).
\end{equation}
Combining (\ref{e531}), (\ref{e53}) and (\ref{e532}), we arrive at
\begin{equation}\label{e533}
|\bd_n u_n(x)| \le C\left(\sqrt{2e_n} + (\ep+\de_n)^\frac12\exp\frac{\si}2(|t| - T_n)\right).
\end{equation}

Let $\ga_n(\cdot) = u_n(\cdot, \th):[-T_n, T_n]\to M$ be the curve given by $u_n$ for fixed $\th$. Then the estimate above shows that
\begin{equation*}
|\dt{\ga_n}(t)| \le |\bd_n u_n(x)| \le C\left(\sqrt{2e_n} + (\ep+\de_n)^\frac12\exp\frac{\si}2(|t| - T_n)\right).
\end{equation*}
Integrating over $[-T_n, T_n]$, we obtain
\begin{equation*}
Length(\ga_n) = \int_{-T_n}^{T_n}|\dt{\ga_n}(t)|dt \le C\left(T_n\sqrt{e_n} + \frac{1}{\si}(\ep+\de_n)^\frac12\right).
\end{equation*}
By taking $n\to \infty$, we find that the length of $\ga_n$ converges to zero and hence the neck is a single closed orbit.
\end{proof}

\begin{lem}\label{l6}
If $0<\nu<\infty$, then for any fixed $t\in [-T_n, T_n]$, we have
\begin{equation}\label{55}
  \lim_{n\to\infty}T_n|\pn u_n| = 0
\end{equation}
and
\begin{equation}\label{54}
  \lim_{n\to\infty}T_n|\p_t u_n| = \frac{\nu}{\sqrt{2\pi}}.
\end{equation}
\end{lem}
\begin{proof}
By the arguments in the proof of Lemma~\ref{l7}, we have pointwise estimate (\ref{e533}) for $|\bd_n u_n|$. In fact, by a
bootstrapping technique(cf. \cite{S}), we may improve the $\ep$-regularity to hold for all higher derivatives of $u_n$. Namely, it is easy to show that for any integer $k\ge 1$, we have
\begin{equation*}
|\bd_n^k u_n(x)| \le C_k\left(\sqrt{2e_n} + (\ep+\de_n)^\frac12\exp\frac{\si}2(|t| - T_n)\right),
\end{equation*}
where $C_k$ is a constant depending on $k$.

Multiplying the above inequality by $T_n$, we get
\begin{equation}\label{e51}
T_n|\bd_n^k u_n(x)| \le C_k\Big(T_n\sqrt{2e_n} + T_n(\ep+\de_n)^\frac12 \exp\frac{\si}2(|t| - T_n)\Big).
\end{equation}
Since $\lim_{n\to \infty} T_n = \infty$ and $\si$ is uniform, it follows that
\begin{equation}\label{e52}
\lim_{n\to \infty}T_n\exp\frac{\si}2(|t|+1 - T_n) = 0.
\end{equation}
Thus, if we define the function
\[ v_n(t, \th) := T_nu_n(t, \th), \]
then (\ref{e51}) and (\ref{e52}) implies that for sufficiently large $n$,
\begin{equation*}
\norm{\bd_n^k v_n}_{C^0(\C_n)} \le C_k(\sqrt{2}\nu + 1).
\end{equation*}
On the other hand, $\bar{A}_n$ converges to the limit connection $A_\infty$. It follows that $\bar{D}_n$ is equivalent to the standard Levi-Civita connection
$\nabla$, which implies
\begin{equation*}
\norm{\nabla^k v_n}_{C^0(\C_n)} \le C_k(\sqrt{2}\nu + 1) + C(A_\infty)
\end{equation*}
where the constant $C(A_\infty)$ only depends on $A_\infty$. Hence $v_n$ converges to some $v_\infty$ in $C^k_{loc}(\Real^1\times S^1)$ for any $k\ge
1$. Since $u_n$ satisfies equation~(\ref{e12}), it is obvious that $v_n$ satisfies
\begin{equation*}
\p_t^2 v_n + \pn^2 v_n + \frac{1}{T_n}\Ga(u_n)(\bd_n v_n, \bd_n v_n) = T_n f_n.
\end{equation*}
Recall that $\norm{f_n}_{L^\infty}$ decays exponentially. Thus, by taking $n \to \infty$, we obtain
\begin{equation}\label{e9}
\p_t^2 v_\infty + \p_{\th, \al_\infty}^2 v_\infty = 0
\end{equation}
where $\p_{\th, \al_\infty} = \p_\th +\al_\infty$ is the derivative induced by $A_\infty$. However, (\ref{e53}) implies that
\[ \p_{\th, \infty} v_\infty = \lim_{n\to \infty}T_n \pn u_n = 0. \]
It follows that $\p_{\th, \al_\infty}^2 v_\infty = 0$ and hence $\p_t^2 v_\infty = 0$ by (\ref{e9}). Thus $\p_t v_\infty(\cdot,\th)$ is independent of $t$ for fixed $\th\in [0,2\pi]$. On the other hand, we have
\begin{align*}
 \p_\th|\p_t v_\infty|^2 &= 2(\p_\th\p_t v_\infty,\p_t v_\infty) = 2(\p_t \p_\th v_\infty,\p_t v_\infty)\\
 & = -2(\p_t (\al_\infty\cdot v_\infty),\p_t v_\infty)= -2(\al_\infty\cdot\p_t v_\infty,\p_t v_\infty)\\
 & = 0.
 \end{align*}
The last identity uses the fact that $\al_\infty$ is skew-symmetric. Therefore, $|\p_t v_\infty|$ is constant all over the cylinder, which implies that
\begin{equation*}
  \lim_{n\to\infty}T_n|\p_t u_n| = \lim_{n\to\infty}\frac{\nu}{\sqrt{e_n}}|\p_t u_n| = \frac{\nu|\p_t v_\infty|}{\sqrt{\int_{\{0\}\times S^1}|\p_t v_\infty|^2 d\th}}=\frac{\nu}{\sqrt{2\pi}}.
\end{equation*}
\end{proof}

\begin{rem}\label{r2}
It can be verified that for any fixed $t$, we have
\begin{equation}\label{e55}
  \lim_{n\to \infty} T_n^2|\pn^2 u_n| = 0.
\end{equation}
Actually, we can follow the same method as in the proof of Lemma~\ref{l1} to prove that the quantity
$$\Th_1(t) :=\int_{S^1}|\pn^2 u_n|^2d\th$$
also decays exponentially along the cylinder. To do this, one only has to apply the Poincar\'e inequality to show that
$\Th_1(t)$ satisfies a similar equation as (\ref{ee1}).
\end{rem}

Now we are in the position to investigate the geometry of the neck. Define the re-parameterized map
$$w_n(s, \th) := u_n(T_n s, \th)$$
on the fixed cylinder $\C_1 = [-1, 1]\times S^1$. Obviously,
\begin{equation}\label{e59}
\lim_{n\to \infty}\pn w_n = \lim_{n\to \infty}\pn u_n = 0.
\end{equation}
It follows from the fact $\lim_{n\to \infty}\al_n = \al_\infty$ that
\[ |\p_\th w_n| \le |\pn w_n| + |\al_n\cdot w_n| \le C. \]
On the other hand, by Lemma~\ref{l6}, we have
$$ \lim_{n\to \infty}|\p_s w_n| = \lim_{n\to \infty}T_n|\p_t u_n| = \frac{\nu}{\sqrt{2\pi}}.$$
Therefore $w_n$ converges to a map $w$ in $C^0(\C_1)$ .

Moreover, using the equation (\ref{e12}), we have
\begin{equation*}
\begin{aligned}
  \p_s^2 w_n &= T_n^2 \p_t^2 u_n\\
  &= -T_n^2(\pn^2u_n + \Ga(u_n)(\bd_n u_n, \bd_n u_n) + f_n)\\
  &= -\Ga(w_n)(\p_s w_n, \p_s w_n) -T_n^2(\pn^2u_n + \Ga(u_n)(\pn u_n, \pn u_n) + f_n).
\end{aligned}
\end{equation*}
By the exponential decay of $f_n$, we have
\[ \lim_{n\to \infty} T_n^2 f_n = 0. \]
By Lemma~\ref{l6}, we have
\[ \lim_{n\to \infty} T_n^2\Ga(u_n)(\pn u_n, \pn u_n) = 0 \]
In view of Remark~\ref{r2}, we also have
\begin{equation}\label{e56}
 \lim_{n\to \infty} T_n^2 \pn^2u_n = 0.
\end{equation}
Therefore, we obtain
\begin{equation}\label{e57}
\lim_{n\to \infty} \(\p_s^2 w_n + \Ga(w_n)(\p_s w_n, \p_s w_n)\) = 0.
\end{equation}

Note that (\ref{e56}) implies
\[ |\p_\th^2 w_n| \le |\pn^2 w_n| + 2|\al_n\cdot \pn w_n| + |\al_n^2\cdot w_n| \le C. \]
and (\ref{e57}) implies
\[ |\p_s^2 w_n| \le |\Ga(w_n)||\p_s w_n|^2 + 1 \le C. \]
Consequently, $w_n$ actually converges in $C^1(\C_1)$ to the limit map $w$.

It follows from (\ref{e59}) that
\begin{equation}\label{e555}
\p_{\th, \al_\infty} w = \lim_{n\to \infty}\pn w_n = 0,
\end{equation}
and from (\ref{e57}) that $w$ satisfies the equation
\begin{equation}\label{e58}
 \p_s^2 w + \Ga(w)(\p_s w, \p_s w) = 0,
\end{equation}
in the weak sense. Hence by the standard elliptic estimates, $w$ is a smooth map. In fact, in view of (\ref{e555}) and (\ref{e58}), $w$ is a twisted geodesic (see Definition~\ref{d:twisted-geodesic}). In other words, for any fixed $\th\in S^1$, the curve $\ga_\th := w(\cdot, \th)$ is a geodesic in $M$, while for any fixed $t$, the curve
$c_t : = w(t, \cdot)$ is an orbit generated by $\al_\infty \in \g$ in $M$. Moreover, the length of the geodesic is given by
\[ Length(\ga_\th) = \lim_{n\to \infty} \int_{- T_n}^{ T_n} |\p_t u_n| dt = \lim_{n\to \infty} \frac{2}{\sqrt{2\pi}}\sqrt{e_n}T_n = \frac{2}{\sqrt{2\pi}}\nu. \]

Finally, if the limit $\nu = \infty$, we may choose $T_n' = \frac{\sqrt{2\pi}}{2\sqrt{e_n}}$ such that
$$\nu' = \lim_{n\to \infty}\sqrt{e_n}T_n'=1.$$
Then the arguments above shows that on a sub-cylinder of length $2T_n'$, the images converge to a twisted geodesic of length $\frac{2}{\sqrt{2\pi}}$. Since $T_n/T_n' \to \infty$, we may find infinitely many such sub-cylinders on $\C_n$. Thus we obtain an infinitely long twisted geodesic.

To summarize, we have proved the following result in non-degenerate case.

\begin{thm}\label{t:non-deg}
  Suppose $\{(A_n, u_n)\}$ satisfies the hypothesis of Theorem~\ref{t:energy-identity}. If there is no energy concentration and $A_n$ is non-degenerating, then the following hold.
  \begin{enumerate}
    \item If $0<\nu <\infty$, then $u_n(\C_n)$ converges to a twisted geodesic of length $\frac{2}{\sqrt{2\pi}}\nu$;
    \item If $\nu = 0$, then the $u_n(\C_n)$ converges to a single closed orbit.
    \item If $\nu=\infty$, then the neck contains an infinitely long twisted geodesic.
  \end{enumerate}
\end{thm}

\subsubsection{The degenerate case}\label{s4}

Finally, let us try to understand the geometry of the neck in the degenerate case, which could be more complicated than the non-degenerate case above. More precisely, suppose the connections $A_n=\al_n d\th$ converges to $A_\infty = \al_\infty d\th$ and the convergence is degenerating. Denote by
$$\rho_n := |\al_n - \al_\infty|$$
and
$$\be_n := (\al_n-\al_\infty)/\rho_n.$$
Since $|\be_n| = 1$ and the Lie algebra $\g$ is finite dimensional, we may suppose $\be_n$ converges to a limit $\be_\infty\in \g$. The quantities $\rho_n$ reflects the degenerating speed of $A_n$ and $\be_\infty$ shows the direction along which $A_n$ degenerates.

In this case, there are three kinds of degenerations involved. Namely, the degeneration of the metric, the blowing-up of the map and the degeneration of the connection. Intuitively, the degeneration of the metric $g_n$ corresponds to the formation of the node, which, by the conformal change, is reflected by the length $T_n$ of the cylinder $\C_n$. The information of the blowing-up of the maps $u_n$ on the cylinder is essentially contained in the quantities $e_n$ defined by (\ref{e46}). If the connection behaves well, i.e. the connection is non-degenerating, the limits $\mu$ and $\nu$ defined before accounts for the competing of $T_n$ and $e_n$, which arise in the energy identity and the length formula. The energy identity in Theorem~\ref{t:energy-identity} also gives us a clue that the degenerating speed of the connection can be measured by $\rho_n = |\al_n - \al_\infty|$.

What is the geometric influence brought by the degeneration of the connection $A_n$? A good illustration can be found in \cite{MT}, where the authors discussed in detail the moduli space of twisted holomorphic curves in the special case where the Lie group is simply $S^1$. Recall that the twisted holomorphic curves are just minimal YMH fields which satisfy the equation
\begin{equation}\label{e:thc}
  \left\{
  \begin{aligned}
  &\bar{\partial}_A\phi = 0,\\
  &\iota_v F_A + \mu(\phi) = c.
  \end{aligned}
  \right.
\end{equation}
Here all the notations agrees with our previous setting except that $\bar{\partial}_A$ is the d-bar operator induced by the connection $A$ and $v$ is a volume form on the base manifold $\Si$. Note that the Lie algebra of $S^1$ is simply $i\Real$, hence a flat connection $A$ on a long cylinder has the form $i\al d\th$ in the temporal gauge, where $\al \in \Real$ is a real number.

In the compactification of moduli space of holomorphic curves in the classical Gromov-Witten theory, there is no neck between the bubbles. Or, in other words, the neck shrinks to a point. However, when the connection (or gauge) comes into the game, the holomorphic curve should be replaced by the twisted holomorphic curves and new phenomenons appear.

The appearance of the connection $A$ is due to the Hamiltonian action of the Lie group $G$ on the manifold $M$. So it not surprising that when the connection is non-degenerating, the neck is no longer a point, but a single orbit of that point, which is generated by the group action induced by $A$. However, if the connection is degenerating, it can be shown that the neck converges to a curve $\ga$ which satisfies the following equation
\begin{equation}\label{e60}
\ga_t = -k(t)Ji\cdot \ga,
\end{equation}
where $k$ is a real function and the dot denotes the action of $i$ on $\ga$. If we denote by $h$ the corresponding Hamiltonian induced by $i$, then the above equation can be rewrite as
\begin{equation*}
\ga_t = k(t)\nabla h(\ga).
\end{equation*}
Therefore the neck turn out to be a gradient line of the Hamiltonian. The proof of the above result in \cite{MT} is based on the following key observations: 1) the angular derivative of the map vanishes in the limit; 2) the degeneration of the connection can be described by $A_n- A_\infty$, which corresponds to the term $k(t)i$ and is related to the Hamiltonian.

Keeping the above example in mind, we are led to the following parallel, but more general picture in our current setting. The key observation is that a gradient line which satisfies a first-order equation~(\ref{e60}) should be replaced by a geodesic in our second-order setting. Indeed, we establish the following result which shows a new geometric phenomenon that the neck turns out to be a closed orbit of a geodesic with potential. There should be some interesting geometric setting where this result can be applied.

More precisely, suppose the connections $A_n$ converges to $A_\infty$ and the convergence is degenerating. Using the notations given in the beginning of this section, we define the limit
\begin{equation*}
\ka := \lim_{n\to \infty}T_n\rho_n
\end{equation*}
Since we don't have any a prior information about the degenerating speed $\rho_n$, the limit $\ka$ can be 0, finite or $+\infty$. For convenience we use $\top$ to denote the projection to the tangent space of $M$.

\begin{thm}\label{t:degenerate}
  Suppose $\{(A_n, u_n)\}$ satisfies the hypothesis of Theorem~\ref{t:energy-identity}. Suppose there is no energy concentration and $A_n$ is degenerating. If the limits $\nu$ and $\ka$ are finite,   then the images of $u_n$ converges to a closed orbit of a perturbed geodesic given by a smooth map $v_\infty:\C_1\to M$ which satisfies 
  \begin{equation}\label{e64}
  \left\{
  \begin{aligned}
  &\p_{\th, \al_\infty} v_\infty=0,\\
  &(\p_s^2 v_\infty + \ka^2\cdot \be_\infty^2 v_\infty)^\top = 0.
  \end{aligned}
  \right.
  \end{equation}
\end{thm}
\begin{proof}
Recall the equation~(\ref{e12}) of $u_n$
\begin{equation}\label{e62}
\bar{D}_n^*\bar{D}_n u_n = (\p_t^2 u_n + \p_{\th, \al_n}^2 u_n)^\top =f_n.
\end{equation}
Moreover, $f_n$ is exponentially bounded by (\ref{e10}) and $\tilde{\Th}_n(t) := \int_{S^1\times\{t\}}|\p_{\th,A_\infty}u_n|^2d\th$ is exponentially bounded by (\ref{e63}). Since by definition $\al_n = \al_\infty + \rho_n \be_n$, we have
\[ \p_{\th, \al_n}^2 u_n = \p_{\th, \al_\infty}^2 u_n + 2\rho_n \be_n\p_{\th, \al_\infty} u_n + \rho_n^2 \be_n^2\cdot u_n.\]
Thus (\ref{e62}) is equivalent to
\begin{equation}\label{e66}
(\p_t^2 u_n + \rho_n^2 \be_n^2\cdot u_n)^\top = -(\p_{\th, \al_\infty}^2 u_n + 2\rho_n \be_n\p_{\th, \al_\infty} u_n)^\top + f_n.
\end{equation}
As in the proof of Lemma~\ref{l6}, we can bound the energy of $u_n$ on a fixed-length cylinder $\C_1(t)$ by
\begin{align*}
  \int_{\C_1(t)}|\bar{D}_n u_n|^2 &= \int_{\C_1(t)}|\p_t u_n|^2 + \int_{\C_1(t)}|\p_{\th, \al_n} u_n|^2\\
  & \le \int_{\C_1(t)}|\p_t u_n|^2 + \int_{\C_1(t)}|\rho_n \be_n \cdot u_n|^2 + \int_{\C_1(t)}|\p_{\th, \al_\infty} u_n|^2\\
  & \le C(e_n + \rho_n^2 + (\ep+\de_n)\exp\si_\infty(|t| - T_n)).
\end{align*}
Hence, by the $\ep$-regularity,
\begin{equation*}
|\bar{D}_n^k u_n| \le C_k\left(\sqrt{e_n} + \rho_n + (\ep+\de_n)^\frac12\exp\frac{\si_\infty}{2}(|t| - T_n)\right).
\end{equation*}
Multiplying both sides by $T_n$, we arrive at
\begin{equation}\label{e65}
T_n|\bar{D}_n^k u_n| \le C_k\left(T_n\sqrt{e_n} + T_n\rho_n + T_n(\ep+\de_n)^\frac12\exp\frac{\si_\infty}{2}(|t| - T_n)\right).
\end{equation}
Obviously, the right hand side of the above inequality is bounded, provided that the limits $\nu$ and $\ka$ are finite.

Now re-parameterize and set $v_n(s, \th) = u_n(T_n s, \th)$. In view of (\ref{e66}), $v_n$ satisfies equation
\begin{equation*}
(\p_s^2 v_n + T_n^2\rho_n^2 \be_n^2\cdot v_n)^\top = -T_n^2(\p_{\th, \al_\infty}^2 v_n + 2\rho_n \be_n\cdot\p_{\th, \al_\infty} v_n)^\top + T_n^2f_n.
\end{equation*}
Clearly, $|\p_s v_n|$ is bounded by (\ref{e65}). Moreover, the terms $\p_{\th, \al_\infty} v_n$, $\p_{\th, \al_\infty}^2 v_n$ and $f_n$ all decays exponentially. Following a similar argument of Section~\ref{s3}, it is easy to see that $v_n$ converges in $C^1$ to a limit map $v_\infty$ which satisfies the equation~(\ref{e64}).
\end{proof}

\begin{rem}\label{r3}
Using the notations above, we can write the energy identity of the degenerate case in Theorem~\ref{t:energy-identity} more precisely.
Namely, if we define
\[ \omega := \lim_{n\to \infty}T_n\rho_n^2, \]
then we have
\begin{align*}
&\lim_{n\to \infty}\int_{\C_n}|(\al_n-\al_\infty)\cdot u_n|^2d\th dt= \lim_{n\to \infty}\int_{-T_n}^{T_n}\int_{S^1}|\rho_n\be_n\cdot u_n|^2 d\th dt\\
&= \lim_{n\to \infty}\int_{-1}^1 \int_{S^1}T_n  |\rho_n\be_n\cdot v_n|^2 d\th ds = \omega \int_{-1}^1\int_{S^1}|\be_\infty\cdot v_\infty|^2d\th ds
\end{align*}
Therefore the identity~(\ref{e13}) becomes
\[ \lim_{n\to \infty}\E(u_n, A_n, \C_n) = 2\omega \int_{-1}^1\int_{S^1}|\be_\infty\cdot v_\infty|^2d\th ds + 2\mu.\]
\end{rem}

\begin{rem}
The neck we get in Theorem~\ref{t:degenerate} is a closed orbit of the curve $\ga_0(\cdot) := v_\infty(\cdot, 0)$. Obviously, $\ga_0$ is a critical point of the energy functional
\[ E(\ga) = \int|d\ga|^2ds + \int (\ga, Q \ga)ds, \]
where $Q= -\ka^2\be_\infty^2$ is a symmetric and non-negative matrix since $\be_\infty$ is skew-symmetric. In other words, the curve is a geodesic with quadratic potential $\int_I (v, Q v)ds$. If the manifold $M$ is the standard sphere, then $\ga_0$ coincides with the famous C. Neumann curve which is used to describe the motion of a charged particle influenced by a magnetic field.
\end{rem}

\begin{rem}
If $\nu+\ka = \infty$, then similar to the proof of Theorem~\ref{t:non-deg} , we can find infinitely many piece of such necks obtained by Theorem~\ref{t:degenerate}. The length of the neck can also by computed. However, a discussion in full generality seems unnecessary.
\end{rem}

Finally, it is easy to see that the neck is actually a twisted geodesic in special cases.

\begin{cor}
  Under the hypothesis of Theorem~\ref{t:degenerate}, if $\ka=0$ or $\be^2\cdot v_\infty = 0$,
  then the images of $u_n$ converges to a twisted geodesic.
\end{cor}

\subsection{Application to twisted holomorphic curves}\label{s:app}

At last, we apply the results obtained above to the spacial case of twisted holomorphic curves. 

More precisely, suppose $(A_n, \phi_n)$ is a sequence of twisted holomorphic curves with bounded YMH energy which satisfy the equation~(\ref{e:thc}). Then by restricting to the collar area which is conformal to the cylinder $\C_n$, we obtain a sequence of YMH fields $(A_n, u_n)$. The results we obtained for general YMH fields all hold for $(A_n, u_n)$ since twisted holomorphic curves are nothing but a special kind of YMH fields. Moreover, in view of (\ref{e:thc}), $(A_n, u_n)$ satisfies the first order equation $\bar{\p}_{A_n} u_n=0$, which in balanced temporal gauge is equivalent to
\begin{equation}\label{e:7}
  \p_t u_n + J\pn u_n = 0.
\end{equation}
Thus it is obvious that the quantity $e_n$ defined by (\ref{e46}) is identically zero and $\mu=\nu=0$. Therefore, if $A_n$ is non-degenerating (which corresponds to the non-critical case in~\cite{MT}), by Theorem~\ref{t:energy-identity}, there is no energy on the neck and by Theorem~\ref{t:non-deg}, the neck shrinks to a single closed orbit. On the other hand, if $A_n$ is degenerating, we may suppose the quantity $\ka$ is finite. Then similar to the proof of Theorem~\ref{t:degenerate}, the re-parameterized maps $v_n(s, \th) = u_n(T_ns, \th)$ converges to a limit $v_\infty$ which satisfies $\p_{\th, \al_\infty}v_\infty = 0$.  Then from (\ref{e:7}) we may deduce that $v_\infty$ satisfies equation
\[ \p_sv_\infty + J\ka\be_\infty \cdot v_\infty = 0. \]
Using the moment map $\mu$, we can define the Hamiltonian with respect to $\be_\infty$ by $h(\cdot) = \<\mu(\cdot), \be_\infty\>$. It follows that $\be_\infty \cdot v_\infty = J\nabla h(v_\infty)$ and hence for any fixed $\th\in S^1$, the curve $\ga_\th(\cdot) :=v_\infty(\cdot, \th)$ turns out to be a gradient line satisfying
\[ \p_s\ga_\th - \ka\nabla h(\ga_\th) = 0. \]
Moreover, by Remark~\ref{r3}, we have the energy identity 
\begin{align*}
\lim_{n\to \infty}\E(u_n, A_n, \C_n) &= 2\omega \int_{-1}^1\int_{S^1}|\be_\infty\cdot v_\infty|^2d\th ds\\
&= -\frac{2\omega}{\ka} \int_{-1}^1\int_{S^1}\<\nabla h(v_\infty), \p_s v_\infty\> d\th ds\\
&= -\frac{2\omega}{\ka} \int_{S^1}h(v_\infty)d\th\Big|_{s=-1}^1 = 0,\\
\end{align*}
since by definition we have $\frac{\omega}{\ka} = \lim_{b\to \infty} \rho_n = 0$. Therefore, there is no energy on the neck in the degenerate case, either. This generalizes the compactness results in \cite{MT} from the special case of $G=S^1$ to arbitrary compact connected Lie group $G$

\section*{Acknowledgements}
Part of this work was carried out when the author was visiting Beijing International Center for Mathematical Research. The author would like
to thank Prof. Gang Tian for his constant support. He would also like to thank Prof. Youde Wang, Yuxiang Li, Miaomiao Zhu and Li Chen for many helpful
discussions.


\vspace{1cm}
\noindent{Chong Song}\\
School of Mathematical Sciences, Xiamen University, Xiamen 361005, P.R. China.\\
Email: songchong@xmu.edu.cn

\end{document}